\documentclass[12pt]{amsart}

\usepackage{amssymb}
\usepackage{lscape}
\usepackage{stmaryrd}

\usepackage{mathpazo}
\usepackage{multicol}
\usepackage{mathrsfs}

\usepackage[all]{xy}
\usepackage[headings]{fullpage}
\usepackage{paralist}

\usepackage{tabularx,ragged2e,booktabs,caption}
\usepackage[colorlinks=true,linkcolor=blue, citecolor=blue, %
                urlcolor=blue, pagebackref=true]{}

\makeatletter

\newcounter{maintheorem}[equation]
\def\themaintheorem{\thesection.\@arabic \c@maintheorem}
\@addtoreset{equation}{section}
\def\theequation{\thesection.\@arabic \c@equation}

\def\theenumi{\@alph\c@enumi}
\def\theenumii{\@roman\c@enumii}

\makeatother

\newtheorem{thm}{Theorem}[section]
\newtheorem{lem}[thm]{Lemma}
\newtheorem{lem-dfn}[thm]{Lemma-Definition}
\newtheorem{prop}[thm]{Proposition}
\newtheorem{cor}[thm]{Corollary}

\theoremstyle{definition}
\newtheorem{defn}[thm]{Definition}

\theoremstyle{remark}

\newtheorem{exam}[thm]{Example}
\newtheorem{rem}[thm]{Remark}

\numberwithin{equation}{section}

\newcommand{\bbZ}{\ensuremath{\mathbb Z}}

\newcommand{\ov}{\overline}

\newcommand{\m}{{\mathfrak m}}

\DeclareMathOperator{\height}{ht}
\DeclareMathOperator{\depth}{depth}

\DeclareMathOperator{\nr}{{nr}}

\DeclareMathOperator{\br}{\bar r}

\DeclareMathOperator{\ord}{ord}
\DeclareMathOperator{\emb}{embdim}

\def\res{{\rm K}}

\def\sers2{{\res[\![x,y]\!]}}
\def\ser3{{\res[\![x,y,z]\!]}}

\def\pol2{{\res[x,y]}}
\def\pol3{{\res[x,y,z]}}

\everymath{\displaystyle}

\newcommand{\FF}{\mathcal{F}}

\newcommand{\MM}{\mathfrak M}

\newcommand{\ol}[1]{\mkern 1.5mu\overline{\mkern-1.5mu#1\mkern-1.5mu}\mkern 1.5mu} 

	\newcommand{\BN}{\mathbb{N}}

\newcommand{\ols}[1]{\mskip.5\thinmuskip\overline{\mskip-.5\thinmuskip {#1} \mskip-.5\thinmuskip}\mskip.5\thinmuskip} 
\newcommand{\olsi}[1]{\,\overline{\!{#1}}} 

\usepackage{xcolor}
\title[Gorenstein associated graded rings of filtrations]{On Gorensteinness of associated graded rings of filtrations}

\author[Bhat]{Meghana Bhat}
\address[Meghana Bhat]{Department of Mathematics, Indian Institute of Technology Dharwad, WALMI Campus, PB Road, Dharwad - 580011, Karnataka, India}
\email{mbhat.math@gmail.com}

\author[Dubey]{Saipriya Dubey}
\address[Saipriya Dubey]{Department of Mathematics, Indian Institute of Technology Dharwad, WALMI Campus, PB Road, Dharwad - 580011, Karnataka, India}
\email{saipriya721@gmail.com}

\author[Masuti]{Shreedevi K. Masuti}
\address[Shreedevi K. Masuti]{Department of Mathematics, Indian Institute of Technology Dharwad, WALMI Campus, PB Road, Dharwad - 580011, Karnataka, India}
\email{shreedevi@iitdh.ac.in}

\author[Okuma]{Tomohiro Okuma}
\address[Tomohiro Okuma]{Department of Mathematical Sciences, 
Yamagata University,  Yamagata, 990-8560, Japan.}
\email{okuma@sci.kj.yamagata-u.ac.jp}

\author[Verma]{Jugal K. Verma}
\address[Jugal K. Verma]{Department of Mathematics, Indian Institute of Technology Bombay,Powai, Mumbai, Maharashtra-400076, India }
\email{verma.jugal@gmail.com}

\author[Watanabe]{Kei-ichi Watanabe}
\address[Kei-ichi Watanabe]{Department of Mathematics, College of Humanities and Sciences, 
Nihon University, Setagaya-ku, Tokyo, 156-8550, Japan and 
Organization for the Strategic Coordination of Research and Intellectual Properties, Meiji University}
\email{watnbkei@gmail.com}

\author[Yoshida]{Ken-ichi Yoshida}
\address[Ken-ichi Yoshida]{Department of Mathematics, 
College of Humanities and Sciences, 
Nihon University, Setagaya-ku, Tokyo, 156-8550, Japan}
\email{yoshida.kennichi@nihon-u.ac.jp}

\thanks{MB is supported by Prime Minister's Research Fellowship (PMRF), Govt. of India. SD is supported by CRG grant CRG/2022/007572, funded by SERB,  Govt. of India. SKM is supported by INSPIRE
faculty award funded by Department of Science and Technology, Govt. of India.  She is also supported by CRG grant CRG/2022/007572 and MATRICS grant MTR/2022/000816 funded by SERB,  Govt. of India. 
The authors MB,SD,SKM and JKV are supported by the SPARC grant P1566.
TO is partially supported by JSPS Grant-in-Aid for Scientific Research (C) Grant Number 21K03215. KW is partially supported by JSPS Grant-in-Aid 
for Scientific Research (C) Grant Number 23K03040. KY is partially supported by JSPS Grant-in-Aid 
for Scientific Research (C) Grant Number 24K06678.}


\date{\today}

\begin{document}

\begin{abstract}
Let $(A, \m)$ be a Gorenstein local ring, 
and $\FF =\{F_n \}_{n\in \mathbb{Z}}$ a Hilbert filtration. In this paper, 
we give a  criterion for Gorensteinness of the associated graded ring 
of $\FF$ in terms of the Hilbert coefficients of $\FF$ in some cases. 
As a consequence we recover and extend a result proved in
\cite{OWY6}.  
Further, we present ring-theoretic properties of the normal tangent 
cone of the maximal ideal of $A=S/(f)$ 
where $S=K[\![x, y_1,\ldots, y_m]\!]$ is a formal power series ring 
over an algebraically closed field $K$, 
and $f=x^a-g(y_1,\ldots,y_m)$, where $g$ is a polynomial with 
$g \in (y_1,\ldots,y_m)^b \setminus (y_1,\ldots,y_m)^{b+1}$, 
and $a, \, b, \, m$ are integers.
We show that the normal tangent cone $\overline{G}(\m)$  
is Cohen-Macaulay if $a \le b$ and $a \ne 0$ in $K$.   
Moreover, we give a criterion of the Gorensteinness of $\overline{G}(\m)$. 
\end{abstract}
    

\maketitle
\section{Introduction}
Let $(A,\m)$ be a Noetherian local ring of dimension $d$ 
with an infinite residue field. 
An $F_1$-good filtration, $\FF = \{F_n\}_{n\in \mathbb{Z}}$ 
is a descending chain $A=F_0 \supseteq F_1 \supseteq \cdots $ 
of ideals satisfying  $F_i F_j \subseteq F_{i+j}$ for all $i,j \in \mathbb{Z}$,  $F_n = A$ for all $n \leq 0$, and $F_{n+1} =F_1F_n$ for all $n \gg 0$. 
In addition, if $F_1$ is an $\m$-primary ideal, 
then $\FF$ is called a {\it Hilbert filtration} (c.f. \cite{Bla97}). 
We write ${G}(\FF):=\bigoplus_{n \ge 0} {F_n}/{F_{n+1}}$ 
for the associated graded ring of $\FF.$
Studying the homological properties of $G(\FF)$, 
in particular, the Cohen-Macaulay property, has been an active topic 
of research since long time, 
see for instance \cite{GN94,Hun87,Lip94}.
\par 
In this paper we study the Gorenstein property of $G(\FF)$ 
with particular attention to the filtration 
$\{\ol{I^n}\}_{n \in \mathbb{Z}}$ where $\ol{J}$ denotes 
the integral closure of $J$. 
The Gorenstein property of $G(\FF)$ when 
$\FF=\{I^n\}_{n \in \mathbb{Z}}$ has been studied extensively 
(eg. \cite{Ooishi, Hyry, GI, JV95, Sally1983, sa, S1}).
\par 
Let $\FF$ be a Hilbert filtration.
It is well known that the Hilbert function of $\FF$, 
$H_{\FF}(n):=\ell(A/F_n)$ coincides with a polynomial 
of degree $d$ for large $n$ (see \cite{Mar89}). 
We denote this polynomial by $P_{\FF}(n)$ and write it as
\[
P_{\FF}(n)=e_0(\FF)\binom{n+d-1}{d}-e_1(\FF)\binom{n+d-2}{d-1}
+\cdots +(-1)^de_d(\FF).
\]
Here $e_i(\FF)$ are integers called the {\it Hilbert coefficients} of $\FF$. 
Recall that $Q$ is said to be a reduction of $\FF$ 
if  $Q\subseteq F_1$ and  $F_{n+1}=QF_n$ for $n\gg0$. 
A reduction $Q$ of $I$ is said to be a minimal reduction 
of $\FF,$ if $Q$ is a reduction of $\FF$ and is minimal 
with respect to the inclusion. 
The reduction number of $\FF$ with respect to a minimal reduction 
$Q$ of $\FF$, 
\[
r_Q(\mathcal{F}):=\min\{r\mid F_{n+1}=QF_n ~\textrm{for all } n\geq r \},
\]
and the {\it reduction number} of $\FF$ is 
\[
r(\FF) :=  \min\{r_Q(\FF) 
\mid Q ~\mathrm{is ~a ~minimal ~reduction ~of} ~\FF \}.
\]
\par 
It is easy to observe that if $I$ is an $\m$-primary ideal in $A$, 
then $\FF=\{I^n\}_{n \in \mathbb{Z}}$ is a Hilbert filtration. 
In this case we write ${G}(I)$ for $G(\FF),~{e}_i(I)$ for $e_i(\FF)$, 
${r}_Q(I)$ for $r_Q(\mathcal{F})$ and ${r}(I)$ for $r(\mathcal{F}).$
\par 
There is another interesting example of a filtration that appears 
in the study of resolution of singularities, 
namely $\FF=\{\ol{I^n}\}_{n \in \mathbb{Z}}$. 
\par 
In \cite{Rees} Rees showed that if $A$ is an analytically unramified
local ring and $I$ is an $\m$-primary ideal in $A$, 
then $\FF = \{\ol{I^n}\}_{n \in \mathbb{Z}}$ is a Hilbert filtration. 
In this case we write $\ol{G}(I)$ for $G(\FF)$. 
We denote $~\ol{e}_i(I)$ for $e_i(\FF)$, $\ol{r}_Q(I)$ for $r_Q(\mathcal{F})$,  and $\ol{r}(I)$ for $r(\mathcal{F})$ 
and call them as the {\it normal Hilbert coefficients}, 
{\it normal reduction number with respect to a minimal reduction $Q$ 
of $\{\ol{I^n}\}_{n \in \mathbb{Z}}$} and {\it normal reduction number},  respectively.
\par 
It has been observed that $G(I)$ and $\ol{G}(I)$ are 
rarely Cohen-Macaulay even if $A$ is Cohen-Macaulay. 
A large attention in the literature is devoted to obtaining criteria for 
$G(\FF)$ to be Cohen-Macaulay in terms of Hilbert coefficients 
and/or reduction number with particular attention 
to the filtration $\{I^n\}_{n \in \mathbb{Z}}$ and 
$\{\ol{I^n}\}_{n \in \mathbb{Z}}$. 
For instance, it is well-known that  if $r(\FF) \leq 1$ in a 
Cohen-Macaulay local ring, then $G(\FF)$ is Cohen-Macaulay 
(c.f. \cite{Mar89}). 
Also, it is known that if $e_1(\FF)=0 $ in a Cohen-Macaulay local ring, 
then $G(\FF)$ is Cohen-Macaulay (c.f. \cite{Mar89}). 
For these reasons, a huge research is devoted to obtaining 
a sharp bound on the Hilbert coefficients and reduction number, 
and studying the behavior of $G(\FF)$ when these bounds 
are extremal. This was also a theme studied 
by several mathematicians, for example, Huneke, Goto, Lipman, 
Sally among others.
\par 
In this paper, we are interested in obtaining 
the Gorenstein property of $G(\FF)$
in terms of the Hilbert coefficients.
As mentioned before this has also a rich history. 
Now, we focus on the content of this paper.
\par 
In \cite{HKU2005} Heinzer, Kim, and Ulrich established a criterion for Gorensteinness of $G(I)$ when $A$ is an Artinian Gorenstein local ring.
This has been extended for Hilbert filtrations in \cite{HKU}.

\par 
They proved that $G(\FF)$ is Gorenstein if and only if 
the $h$-vector of $G(\FF)$ is symmetric.
In this paper we extend this result for any $d$-dimensional Gorenstein local ring (Theorem \ref{thm:GoriffhVectorCondition}). 
We use this result to characterize the Gorensteinness 
of $G(\FF)$ in terms of the Hilbert coefficients.

\par 
A first result in this direction is obtained when $r(\FF) \leq 2$ 
(Theorem \ref{thm:GorRedu12}). 
We make an additional assumption when $r(\FF)=2$, namely: 
there exists a minimal reduction $Q$ of $\FF$ such that $r_Q(\FF)=2$ 
and $F_2 \cap Q = QF_1.$ This condition ensures that $G(\FF)$ is 
Cohen-Macaulay (Lemma \ref{lemma:GFCM}), 
and helps to compute the Hilbert series of $G(\FF)$.  
Notice that this condition is not restrictive. 
For instance, the filtrations $\FF=\{\ol{I^n}\}_{n \in \bbZ}$ 
and $\FF=\{\m^n\}_{n \in \bbZ}$ satisfy this criterion.
As a consequence of our result, we obtain a criterion 
for Gorensteinness of $\ol{G}(I)$ (Corollary \ref{cor:red12_normal}). 
This recovers and extends a recent result by Okuma, Watanabe, 
and Yoshida in \cite{OWY6}.

\par 
In Section \ref{sec:bound on rel red}, 
we next investigate the case $r(\FF) >2$.  
A difficulty in this case is the computation of the Hilbert series of 
$G(\FF)$.  
In order to compute this, we introduce   
the notion of {\it maximal relative reduction number}
for Hilbert filtration $\FF$.  
This is inspired by the definition of relative reduction number 
for normal filtrations defined in \cite{OWY2019} in dimension two. 
\par
The {\it relative reduction number} of $\FF$ with respect to a minimal
reduction $Q$ of $\FF$ is defined as
\[
\nr_Q(\mathcal{F}):=\min \{r\in \mathbb{Z}_+\mid F_{r+1}=QF_r \}.
\]
The {\it relative reduction number} of $\FF$ is defined as
\[
\nr(\FF): = \min\{\nr_Q(\FF) 
\mid Q ~\mathrm{is ~a ~minimal ~reduction ~of} ~\FF \}.
\] 
\par 
Observe that $\nr(\FF) \leq r(\FF)$ and the inequality can be strict 
in general (see \cite[Example 2.8]{ORWY2022}).  
If $\mathcal{F}=\{I^n\}_{n \in \bbZ},$ then $\nr_Q(\FF) =r_Q(\FF)$ 
for any minimal reduction $Q$ of $I$ and hence $\nr(F)=r(\FF)$. 
For $\FF=\{\ol{I^n}\}_{n\in \mathbb{Z}},$ we denote $\nr_Q(\FF)$ 
by ${\nr}_Q(I)$ and $\nr(\FF)$ by ${\nr}(I)$ 
and the latter is called the {\it  relative normal reduction number}.
Note that $\nr(I)$ does not refer to the relative reduction number of the filtration $\mathcal{F}=\{I^n\}_{n \in \bbZ}$.

\par 
In Theorem \ref{thm:nr(F)=r(F)} 
we derive an upper bound for $\nr(\FF)$ in a $d$-dimensional 
Cohen-Macaulay local ring when $\depth G(\FF) \geq d-1$.
We prove that
\begin{equation}\label{nrMax}  
\nr(\mathcal{F})\leq {e_1}(\mathcal{F})
-e_0(\mathcal{F})+\ell(A/F_1)+1.
\end{equation}
\par 
In \cite{ORWY2022} the inequality \eqref{nrMax} was proved 
for $\FF=\{\ol{I^n}\}_{n \in \mathbb{Z}}$ in a 
two-dimensional normal excellent local domain.
We also examine the conditions under which $\nr(\FF)$ achieves 
the upper bound, which serves as a crucial technical assumption 
for the computation of the Hilbert series of $G(\FF)$.  
We say $\FF$ has {\it maximal relative reduction number}
if $\nr(\FF) = e_1(\FF) -e_0(\FF)+\ell(A/F_1) +1$. 
In this case, we are able to compute the Hilbert series of $G(\FF)$ 
(see Theorem \ref{thm:HSMaximal}) 
and obtain a critera for $G(\FF)$ to be Gorenstein 
when $r(\FF) \ge 2$ (see Theorem \ref{thm:GorMaximal}). 
These criteria stated above enable us to understand 
the Gorenstein property of $\ol{G}(I)$ even without having 
the information on the explicit expression of $\ol{I^n}$. 

\begin{thm}[See Theorems \ref{thm:HSMaximal}, \ref{thm:GorMaximal}] \label{Intro-thm}
Let $(A,\m)$ be a Cohen-Macaulay local ring of dimension $d \ge 1$ 
and $\FF=\{F_n\}_{n \in \bbZ}$ a Hilbert filtration. 
Suppose that $r=\nr(\FF) \ge 2$ is maximal. 
Then 
\begin{enumerate}
\item[(a)] The Hilbert series of $G(\FF)$ is 
\[
HS_{G(\FF)}(t)=\dfrac{\lambda+(e_0(\FF)-\lambda-1)t+t^r}{(1-t)^d}, \text{where}\; \lambda=\ell(A/F_1).  
\]
\item[(b)] $G(\FF)$ is Gorenstein if and only if $G(\FF)$ is Cohen-Macaulay and $F_1=\m$ $($resp. $F_1=\m$ and $e_0(\m)=2$$)$ 
if $r=2$ $($resp. $r \ge 3$$)$. 
\end{enumerate}
\end{thm}

\par 
We include some interesting examples of semigroup rings 
in dimension one in Section \ref{sec:Examples} 
and use our criteria to detect Gorensteinness of $\ol{G}(\m)$. 

\par 
In Section \ref{sec:GorGm} and Section \ref{sec:embdimGm} 
we study Gorensteinness of $\ol{G}(\m)$ in hypersurface rings 
$A=K[\![x,y_1,\ldots,y_m]\!]/(f)$.
The standard graded ring defined by
\[
G(\m)=\bigoplus_{n \ge 0} \m^{n}/\m^{n+1}
\]
is called the \textit{tangent cone} of $\m$.
It is well-known that for a local ring $A$,
$G(\m)$ is a hypersurface (and thus Gorenstein) with 
multiplicity of $G(\m)$ equal to $\ord(f)=r_Q(\m)+1$.
\par 
A criterion for Gorensteinness of $\overline{G}(I)$ in terms of 
$\br_Q(I)$ and the other geometric properties is given 
in \cite{OWY6, OWY7}. 
In \cite[Example 2.8]{ORWY2022} an explicit example of 
an $\m$-primary integrally closed ideal $I \subset A$ 
is given such that $\br_Q(I) > \nr_Q(I)$. 
Such an ideal $I$ gives an example for which 
$\overline{G}(I)$ is \textit{not} Cohen-Macaulay. 
However, the ring $A$ is \textit{not} Gorenstein in this example. 
So it is natural to ask the following: 

\par \vspace{2mm} \par \noindent 
{\it Question.} 
Let $A$ be a reduced hypersurface. Then 
\begin{enumerate}
\item[\rm{(a)}] Is $\overline{G}(\m)$ a hypersurface?
\item[\rm{(b)}] Is $\overline{G}(\m)$ Gorenstein?
\item[\rm{(c)}] Is $\overline{G}(\m)$ Cohen-Macaulay?
\end{enumerate}

\par \vspace{2mm}
This Question (b)  has a negative answer. 
Surprisingly, there exist many examples of Brieskorn hypersurfaces 
for which $\overline{G}(\m)$ is NOT Gorenstein (see \cite{OWY6}). 
However, we have no negative answers to Question (c). 
So the main aim of Sections \ref{sec:GorGm} and 
\ref{sec:embdimGm} is to prove the Cohen-Macaulayness 
of $\overline{G}(\m)$ for certain hypersurfaces, 
and to give a criterion for Gorensteinness for such $\overline{G}(\m)$.  
Furthermore, we compute $\br_Q(\m)$ in this case.

\par \vspace{2mm}
Let $K$ be an algebraically closed field, and 
let $m \ge 1$, $2 \le a \le b$ be integers, 
and put $S=K[\![x, y_1,\ldots, y_m]\!]$ be a formal power series ring
with $(m+1)$-variables over $K$.  
Put $\underline{y}=y_1,\ldots,y_m$. 
Let $g$, $f$  be polynomials such that  
\begin{eqnarray*}
g &= & g(\underline{y}) \in (\underline{y})^b 
\setminus (\underline{y})^{b+1}, \\
f &=& f(x,\underline{y})=x^a-g(\underline{y}).
\end{eqnarray*}
Then a hypersurface $A=S/(f)$  
is called \lq\lq Zariski-type hypersurface'' if $a \ne 0$ in $K$ 
and $a \le b=\ord_{(\underline{y})}(g)$. 
\par \vspace{2mm}
The main result of these sections is the following:

\par \vspace{2mm} \par \noindent 
\begin{thm}[See Theorems \ref{S1-Main}, \ref{Gor-Main}, Propositions \ref{Hyp_Maximal}, \ref{MaxEmb}] \label{thm:GGorhyp}
Let $m \ge 1$, $2 \le a \le b$ be integers.  
Let $A=K[[x_0,y_1,\ldots,y_m]]/(x^a -g(\underline{y}))$ be 
a Zariski-type hypersurface of $b=\ord_{(\underline{y})}(g)$ 
and $\m$ its unique maximal ideal. 
Then 
\begin{enumerate}
\item[\rm{(a)}] $\br_Q(\m)=\lfloor \frac{(a-1)b}{a} \rfloor$. 
\item[\rm{(b)}] $\overline{G}:=\overline{G}(\m)$ is Cohen-Macaulay with 
$\dim \overline{G}=m$ and $e_0(\overline{G})=a$. 
\item[\rm{(c)}] $\overline{G}$ is Gorenstein if and only if $b \equiv 0$ 
or $d \pmod{a}$, where $d =\gcd(a,b)$. 
\item[\rm{(d)}] $\nr(\m)$ is maximal if and only if $(a,b)=(2,n),(3,3),(3,4),(3,5)$, where $n \ge 2$. 
\item[\rm{(e)}] 
$\overline{G}$ has maximal embedding dimension if and only if either $(i)$ $a=2$ or $(ii)$ $b \equiv a-1 \pmod{a}$. 
\end{enumerate}
\end{thm}

\par   
We refer \cite{BH98} for undefined terms.
\section{Gorensteinness of $G(\FF)$ in small reduction numbers}
\label{sec:GF Gor dim1}
\par 
In this section, we discuss the Gorenstein property of $G(\FF)$ of a Hilbert filtration $\FF$ in the case of $r(\FF) \le 2$. 
\par 
In \cite{HKU} Heinzer, Kim and Ulrich gave criteria for 
$G(\FF)$ to be Gorenstein for an Artinian Gorenstein local ring. 
In the following theorem, we extend this result for any 
$d$-dimensional Gorenstein local ring. 
This plays a very important role in determining the Gorenstein 
property of the associated graded ring in terms of 
the Hilbert coefficients of $\FF$.

\begin{thm} \label{thm:GoriffhVectorCondition}
Let $(A,\m)$ be a $d$-dimensional Gorenstein local ring, $\mathcal{F}=\{F_n\}_{n\in \mathbb{Z}}$ a Hilbert filtration. 
Suppose $G(\FF)$ is Cohen-Macaulay. 
Let
\[
HS_{G(\mathcal{F})}(z)=\frac{h_0+h_1t+\cdots+h_st^s}{(1-t)^d} \;\;
(h_s \ne 0)
\]
be the Hilbert series of $G(\FF)$.
Then $G(\FF)$ is Gorenstein if and only if $h_{s-i}=h_i$ for all $i=0,1,\ldots ,s$. 
\end{thm}

\begin{proof}
Let $Q$ be a minimal reduction of $\FF$. 
By \cite[Theorem 8.6.3]{HS06} there exists a generating set 
$x_1,\ldots,x_d$ of $Q$ such that $x_1,\ldots,x_d$ is a superficial sequence. 
Since $G(\FF)$ is Cohen-Macaulay, 
$x_1^*,\ldots x_d^*$ is a $G(\FF)$-regular sequence 
by \cite[Lemma 2.1]{HucMar} where $x_i^*=x_i+F_2$ 
for $i=1,2,\ldots ,d$. 
Therefore $G(\FF)/(x_1^*,\ldots ,x_d^*)\cong G(\FF/Q)$. 
Since $x_1^*,\ldots, x_d^*$ is a regular sequence in $G(\FF)$,
\[
HS_{G(\FF)/(x_1^*,\ldots, x_d^*)}(t)=h_0+h_1t+\cdots+h_st^s.
\]
Hence
\[
HS_{G(\FF/Q)}(t)=h_0+h_1t+\cdots+h_st^s.
\]
Therefore by \cite[Theorem 4.2]{HKU} $G(\FF/Q)$ is Gorenstein 
if and only if  $h_{s-i}=h_i ~\textrm{ for } i=0,1,\ldots,s$.
Now the result follows because $x_1^*,\ldots, x_d^*$ is 
a regular sequence in $G(\FF)$ and 
$G(\FF)/(x_1^*,\ldots ,x_d^*)\cong G(\FF/Q)$. 
\end{proof}

\par 
In general for an arbitrary filtration $\FF$, we have $\nr_Q(\mathcal{F})\leq r_Q(\mathcal{F})$, 
and hence $\nr(\FF) \leq r(\FF)$. In \cite[Example 2.8]{ORWY2022} 
the authors gave an example for which the inequality is 
strict for the filtration $\FF=\{\overline{I^n}\}_{n \in \bbZ}$.  
In the following proposition we show that if $G(\FF)$ is Cohen-Macaulay, 
then $\nr(\FF) = r(\FF)$ for any Hilbert filtration $\FF$. 

\begin{prop}\label{lemma:nr=r}
Let $(A,\m)$ be a Cohen-Macaulay local ring of dimension $d$. 
Let $\mathcal{F}=\{F_n\}_{n\in \mathbb{Z}}$ be a Hilbert filtration. 
If $G(\mathcal{F})$ is Cohen-Macaulay, 
then $\nr_Q(\mathcal{F})= r_Q(\mathcal{F})$ for any minimal reduction $Q$ of $\FF$. 
In particular, $  \nr(\FF) = r(\FF).$
\end{prop}

\begin{proof}
We always have $\nr_Q(\mathcal{F}) \leq r_Q(\mathcal{F})$.  
Let $Q$ be a minimal reduction of $F_1$ such that $\mu(Q)=d$. 
Since $G(\FF)$ is Cohen-Macaulay, by \cite[Proposition 3.5]{HucMar} 
$Q \cap F_n=QF_{n-1}$ for all $n\geq 1$.  
Let $m=\nr_Q(\mathcal{F}).$ Then $F_{m+1}=QF_{m}$. 
Hence $F_{m+2} \subseteq F_{m+1} \subseteq Q$. 
Therefore
\[
F_{m+2} = F_{m+2} \cap Q = QF_{m+1}.
\]
Continuing this successively we get ${F_{k+1}} = {QF_k}$ for all $k\geq m$. 
This implies that $r_Q(\mathcal{F})\leq m=\nr_Q(\mathcal{F})$. 
Thus $r_Q(\mathcal{F})=\nr_Q(\mathcal{F})$. 
\end{proof}

\par     
Our goal is to use Theorem \ref{thm:GoriffhVectorCondition}
to obtain criteria for $G(\FF)$ to be Gorenstein in terms of the Hilbert coefficients. For this, first of all we should know whether  $G(\FF)$ is Cohen-Macaulay. There is a huge research in the literature on obtaining criteria for $G(\FF)$
    to be Cohen-Macaulay in terms of numerical invariants such as Hilbert coefficients and reduction number. Below we present a well-known criteria for $G(\FF)$ to be Cohen-Macaulay in terms of reduction number. We remark that the condition in $(b)$ is not restrictive. For instance, the filtrations $\FF=\{\ol{I^n}\}_{n \in \bbZ}$ and $\FF=\{\m^n\}_{n \in \bbZ}$ satisfy this criterion.

\begin{lem}
\label{lemma:GFCM}
 Let $(A, \m) $ be a $d$-dimensional Cohen-Macaulay local ring and $\FF=\{F_n\}_{n\in \mathbb{Z}}$ a Hilbert filtration. Then $G(\FF)$ is Cohen-Macaulay in each of the following cases:
 \begin{enumerate}
  \item $r(\FF)=1;$
  \item  there exists a minimal reduction $Q$ of $\FF$ such that $r_Q(\FF)=2$ and $F_2 \cap Q = QF_1$.
 \end{enumerate}
\end{lem}
\begin{proof}
The proof follows by \cite[Proposition 3.5]{HucMar}.
\end{proof}
\par 
Now in order to provide conditions for $G(\FF)$ to be Gorenstein, we need to compute the Hilbert series of $G(\FF)$. 
The computation of the Hilbert series of $G(\FF)$ is not an easy task in general.
In the following proposition we compute the Hilbert series of $G(\FF)$ 
in the case $\FF$ satisfies the conditions $(a),$ $(b)$ of 
Lemma \ref{lemma:GFCM}, and under certain additional assumptions 
when $r(\FF) \geq 3.$
For this, we recall the following definition: 
For a Hilbert filtration $\FF,$ the \textit{postulation number} 
of $\FF$ is given by 
\[
n(\FF):=\sup\{n\in \mathbb{Z}|H_{\FF}(n)\neq P_{\FF}(n)\}.
\]
If $\FF=\{\ol{I^n}\}_{n\in \bbZ},$ then $n(\FF)$ is denoted by $\overline{n}(I)$. 
We also recall the result \cite[Corollary 3.8]{Mar89} that will be used frequently in the remainder of the paper.  
Let $(A, \m)$ be a $d$-dimensional Cohen-Macaulay local ring with infinite residue field and $\mathcal{F} = \{F_n\}_{n \bbZ}$ a Hilbert filtration of ideals where $F_1$ is an $\m$-primary ideal. 
Suppose $\depth G(\mathcal{F}) \geq d-1$. Then
\begin{equation}\label{redandpos}
  r({\mathcal{F}}) = n(\mathcal{F}) +d.
\end{equation}
\par 
We can use the following lemma in order to reduce several problems to the case of dimension one. 
Recall that if $f:\bbZ \to \mathbb{N}$ is a numerical function, 
then $\Delta(f):\bbZ \to \bbZ$ is a numerical function defined by $\Delta(f)(n):=f(n)-f(n-1)$.

\begin{lem}[cf. \textrm{\cite[Lemma 3.14]{Mar89}}]  \label{Reduce}
Let $(A,\m)$ be a Cohen-Macaulay local ring of dimension $d \ge 1$ 
and $\FF=\{F_n\}_{n \in \bbZ}$ a Hilbert filtration. 
Let $k$ be an integer with $1 \le k \le d$. 
Take a sequence $x_1,x_2,\ldots,x_d \in F_1$  such that 
\begin{enumerate}
\item[\rm{(i)}] $Q=(x_1,x_2,\ldots,x_d)$ is a minimal reduction of $\FF$. 
\item[\rm{(ii)}] $\underline{x}=x_1,\ldots,x_k$ is a superficial sequence in $F_1$. 
\end{enumerate}
If $\depth G(\FF) \ge k$, then 
\begin{enumerate}
\item[\rm{(a)}] $x_1^{*},\ldots,x_k^{*}$ forms a $G(\FF)$-regular sequence. 
\item[\rm{(b)}] If we put $\overline{\FF}:=\FF/(\underline{x}):=\{F_n+(\underline{x})/(\underline{x})\}_{n \in \bbZ}$, then 
\[
P_{\overline{\FF}}(n)=\Delta^k(P_{\FF}(n))=\sum_{i=0}^{d-k} e_i(\FF)(-1)^i
{n+d-k-1-i \choose d-k-i}.
\] 
In particular, $e_i(\overline{\FF})=e_i(\FF)$ for each $i=0,1,\ldots,d-k$. 
\end{enumerate}
\end{lem}

\begin{prop}\label{HSred12}
Let $(A, \m) $ be a Cohen-Macaulay local ring 
of dimension $d \ge 1$ and $\FF=\{F_n\}_{n\in \bbZ}$ a Hilbert filtration.
Put $\lambda = \ell(A/F_1)$, $e_i=e_i(\FF)$ for $i \ge 0$.
\begin{enumerate}
 \item[\rm{(a)}] If $r(\FF)=1$, then $e_1=e_0-\lambda$, $e_i \ge 0$ for all $i \ge 2$ and 
\[
HS_{G(\FF)}(t)= \frac{\lambda + (e_0-\lambda)t}{(1-t)^d} =  \frac{(e_0-e_1) + e_1t}{(1-t)^d}. 
\]
\item[\rm{(b)}] Suppose there exists a minimal reduction $Q$ of $\FF$ such that $r_Q(\FF)=2$ and $F_2 \cap Q = QF_1$. 
Then $e_i=0$ for $i \ge 3$ and 
\[
HS_{G(\FF)}(t)= 
\frac{\lambda + (2e_0-e_1-2\lambda)t +(\lambda-e_0+e_1)t^2}{(1-t)^d}. 
\]
Also, if $d \ge2$, then $e_2=\lambda-e_0+e_1$ and 
\[
HS_{G(\FF)}(t)= 
\frac{(e_0-e_1+e_2) + (e_1-2e_2)t +e_2t^2}{(1-t)^d}. 
\]
\end{enumerate}
\end{prop}

\begin{proof}
First suppose $d=1$. 
{(a)} By Lemma \ref{lemma:GFCM} $G(\FF)$ is Cohen-Macaulay. 
Hence using Equation \eqref{redandpos} and Proposition \ref{lemma:nr=r}
 we get $n(\FF) = 1-1 = 0$. Hence for all $n \geq 1$, $H_{\FF}(n)=P_{\FF}(n) = e_0n-e_1$. 
Taking $n=1$, we obtain $\lambda= \ell(A/F_1) = e_0 -e_1$.
Also,
$\ell(F_n/F_{n+1})= \ell(A/F_{n+1})-\ell(A/F_n)=P_{\FF}(n+1)-P_{\FF}(n)=e_0$ 
for all  $n \geq 1$.
Therefore
\[
HS_{G(\FF)}(t) = \sum_{n\in \mathbb{Z}} \ell(F_n/F_{n+1})t^n
=\lambda +\sum_{n\geq 1} e_0 t^n = \frac{\lambda + (e_0-\lambda)t}{(1-t)}.
\]
\par \noindent
{(b)} By Lemma \ref{lemma:GFCM} $G(\FF)$ is Cohen-Macaulay. 
Hence using Equation \eqref{redandpos} $n(\FF) = 2-1 =1$. 
In particular, $P_{\FF}(n) = H_{\FF}(n)$ for all $n \geq 2$.  
Taking $n =2$, we obtain $\ell(A/F_2) = 2e_0 -e_1$.
Also, $\ell(F_n/F_{n+1})=P_{\FF}(n+1)-P_{\FF}(n)=e_0$ for all $n \geq 2$.
Thus the Hilbert series of $G(\FF)$ is given by
\begin{equation*}
\begin{split}
 HS_{G(\FF)}(t) 
& =  \sum_{n\in \mathbb{Z}} \ell(F_n/F_{n+1})t^n \\
& = \ell(A/F_1) + (\ell(A/F_2)-\ell(A/F_1))t 
+ \sum_{n\geq 2} \ell(F_n/F_{n+1})t^n \\
& = \lambda +(2e_0-e_1-\lambda)t  + \sum_{n \geq 2} e_0 t^n \\
& = \frac{\lambda + (2e_0-e_1-2\lambda)t +(\lambda-e_0+e_1)t^2}{1-t}. 
\end{split} 
\end{equation*}

\par \vspace{2mm}
Next suppose that $d \ge 2$.  
By Lemma \ref{lemma:GFCM} $G(\FF)$ is Cohen-Macaulay. 
Then one can take a sequence $x_1,x_2,\ldots,x_d$ which satisfies 
the condition of Lemma \ref{Reduce} for $k=d-1$. 
If we put $\overline{A}=A/(\underline{x})$, then $\overline{A}$ is a one-dimensional Cohen-Macaulay local ring and $\overline{\FF}$ is 
a Hilbert filtration. 
Also, $\ell(\overline{A}/\overline{F_1})=\lambda$,  
$e_i(\overline{\FF})=e_i$ for $i=0,1$, and $r(\overline{\FF})=r(\FF)$. 
Indeed, the last assertion follows from the following fact: 
\begin{equation} \label{Red_rf}
\dfrac{\overline{F_{n+1}}}{x_d \overline{F_n}} 
= \dfrac{F_{n+1}+\underline{x}A}{QF_n+\underline{x}A} 
\cong  \dfrac{F_{n+1}}{(QF_n+\underline{x}A)\cap F_{n+1}} 
= \dfrac{F_{n+1}}{QF_n+\underline{x} F_{n}} =\dfrac{F_{n+1}}{QF_n}. 
\end{equation}
Therefore one can reduce the proof to the case of dimension one. 
\par 
Moreover, if $d \ge 2$, then since $n(\FF) = 2-d \le 0$, 
we have $H_{\FF}(1)=P_{\FF}(1)$ and thus $\lambda=e_0-e_1+e_2$. 
Also, $e_i(\FF)=0$ for $i \ge 3$. 
\end{proof}

\par 
Using Proposition \ref{HSred12} we provide criteria for $G(\FF)$ to be Gorenstein in terms of the Hilbert coefficients.

\begin{thm}\label{thm:GorRedu12}     
Let $(A, \m) $ be a Gorenstein local ring of dimension $d \ge 1$.  Let $\mathcal{F} = \{F_n\}_{n \in \mathbb{Z}}$ be a Hilbert filtration.
\begin{enumerate}
 \item[\rm{(a)}] Let $r(\FF) = 1$. Then $G(\FF)$ is Gorenstein if and only if $e_0(\FF) = 2e_1(\FF)$.
 \item[\rm{(b)}] Suppose there exists a minimal reduction $Q$ of $\FF$ such that $r_Q(\FF)=2$ and $F_2 \cap Q = QF_1$. 
Then $G(\FF)$ is Gorenstein if and only if $e_0(\FF) = e_1(\FF)$.
\end{enumerate}
\end{thm}

\begin{proof} 
Put $\lambda=\ell(A/F_1)$ and $e_i=e_i(\FF)$ for all $i \ge 0$. 
In both the cases, by Lemma \ref{lemma:GFCM}  $G(\FF)$ is Cohen-Macaulay. 
\par \noindent 
{(a)}
By Proposition \ref{HSred12}, we have
\[
HS_{G(\FF)}(t)= \frac{(e_0- e_1) + e_1t}{(1-t)^d}.
\] 
Note that $e_1 \neq 0$. Using Theorem \ref{thm:GoriffhVectorCondition}, we obtain $G(\FF)$ is Gorenstein if and only if $e_0 = 2e_1$. 

\par \noindent
{(b)} 
By Proposition \ref{HSred12}, we have
\[
HS_{G(\FF)}(t)=\frac{\lambda + (2e_0-e_1-2\lambda)t + 
(\lambda-e_0+e_1)t^2}{(1-t)^d}.
\]
By Theorem \ref{thm:GoriffhVectorCondition}, we obtain 
$G(\FF)$ is Gorenstein if and only if $e_0=e_1$. 
\end{proof}

\par 
Now we recover the result in \cite[Theorem 2.1]{OWY6}. Moreover, we extend it for any $d$-dimensional analytically unramified Gorenstein local ring. 

\begin{cor}
\label{cor:red12_normal}
Let $(A, \m) $ be a $d$-dimensional analytically unramified Gorenstein 
local ring of dimension $d \geq 1$, and $I$ an $\m$-primary ideal 
in $A$.  
\begin{enumerate}
\item[\rm{(a)}] If $\overline{r}(I) = 1$, then $\olsi{G}(I)$ is Gorenstein if and only if $e_0(I) = 2\ol{e}_1(I)$.
\item[\rm{(b)}] If $\overline{r}(I) = 2$, then $\olsi{G}(I)$ is Gorenstein if and only if $e_0(I) = \ol{e}_1(I)$.
\end{enumerate}
\end{cor}

\begin{proof}
(a) If $\overline{r}(I) =1$, then the result follows by Theorem \ref{thm:GorRedu12}. 
\par \noindent 
(b) Let $\overline{r}(I) = 2$. 
By \cite[Theorem 1]{It},  we obtain $Q \cap \ol{I^2} = Q\ol{I}$. 
By Theorem \ref{thm:GorRedu12}, the result follows.
\end{proof}

\section{A Bound on relative reduction number}
\label{sec:bound on rel red}
In this section we obtain an upper bound for 
the relative reduction number 
$\nr(\FF)$ of $\FF$ when the $\depth$ of $ G(\FF)$ is almost maximal.

\par 
In general, the reduction number $r_Q(\FF)$ of $\FF$ 
and the relative reduction number $\nr_Q(\FF)$ of $\FF$ are 
dependent on the minimal reduction $Q$. 
It is well-known that if $\depth{G(\FF)} \geq d-1$, 
then $r_Q(\FF)$ is independent of the minimal reduction $Q$.
In the following proposition we show that $\nr_Q(\FF)$ is also 
independent of $Q$ if $\depth{G(\FF)} \geq d-1$.

\begin{prop}\label{prop:relredind}
Let $(A, m)$ be a $d$-dimensional Cohen-Macaulay local ring, 
and $\mathcal{F}=\{F_n\}_{n\in \mathbb{Z}}$ a Hilbert filtration. 
If $\depth{G(\FF)} \geq d-1$, then $\nr_Q(\FF)$ is independent 
of any minimal reduction $Q$.
\end{prop}

\begin{proof}
Using \cite[Theorem 3.6]{Mar89} for all $n \in \mathbb{Z},$ we have
\[
\Delta^d(P_{\FF}(n) - H_{\FF}(n)) = \ell(F_{n+d}/QF_{n+d-1}).
\] 
Since the expression on the left-hand side is independent of $Q$, we get $\ell(F_{n+d}/QF_{n+d-1})$ is independent of $Q$. Hence it is easy to verify that $\nr_Q(\FF)$ is independent of $Q.$ 
\end{proof}

\par
If $(A,\m)$ is a Cohen-Macaulay local ring of dimension two and $\FF=\{\overline{I^n}\}_{n \in \bbZ},$ then $\depth G(\FF) \geq 1$ by \cite[Proposition 3.25]{Mar89}. Hence by Proposition \ref{prop:relredind} ${\nr}_Q(I) $ is independent of $Q$ and thus ${\nr}(I)= {\nr}_Q(I).$ This is a reason that the authors use the notation ${nr}(I)$ (without suffix) for the relative normal reduction number in \cite{ORWY2022}.

\par 
In the next theorem we extend a result by Okuma, Rossi, Watanabe, and Yoshida. In \cite[Theorem 2.7]{ORWY2022} they showed that if $(A,\m)$ is an excellent two-dimensional normal local domain containing an algebraically closed field $A/\m,$ then for an $\m$-primary ideal $I$
\[
\nr(I) \leq {\ov{e}_1}(I)-\ov{e}_0(I)+\ell(A/\ov{I})+1.
\]
They also provided a criterion for the equality to hold. 
In the following theorem we prove an analogue bound for 
any Hilbert filtration in a $d$-dimensional Cohen-Macaulay 
local ring, provided $G(\FF)$ has almost maximal depth.

\begin{thm} \label{thm:nr(F)=r(F)}
Let $(A, \m)$ be a $d$-dimensional Cohen-Macaulay local ring, 
and $\mathcal{F}=\{F_n\}_{n\in \mathbb{Z}}$ a Hilbert filtration. 
Suppose $\depth{G(\mathcal{F})}\geq d-1$. Then  
\begin{enumerate}
\item[\rm{(a)}]
$\nr(\mathcal{F})\leq {e_1}(\mathcal{F})-e_0(\mathcal{F})+\ell(A/F_1)+1$.
\item[\rm{(b)}] The equality holds in (1) 
if and only if $\nr(\FF) = r(\FF)$ and 
$\ell(F_n/QF_{n-1}) = 1$ 
for all $2 \leq n \leq \nr(\FF)$ for any minimal reduction $Q$ of 
$\FF$. 
\end{enumerate}
When this is the case, $e_i(\FF) = \binom{\nr(\FF)}{ i}$ for $2 \leq i \leq d$.
\end{thm}

\par 
A Hilbert filtration $\FF$ is said to have {\it maximal relative reduction number} (or, $\nr(\FF)$ is {\it maximal}) if $\nr(\FF)=e_1(\FF)-e_0(\FF)
+\ell(A/F_1)+1$. 

\begin{proof}
(a) Using \cite[Proposition 4.6]{HucMar}, for any minimal reduction $Q$ of $\FF$ we have, 
\begin{eqnarray} \label{Eqn:e1}
e_1({\mathcal{F}}) & = &\sum_{n \geq 1}^{r(\mathcal{F})} \ell(F_n/QF_{n-1}) \nonumber\\
 & =& \ell(F_1/Q) + \sum_{n \geq 2}^{r(\mathcal{F})} \ell(F_n/QF_{n-1}) \nonumber \\
 & = &e_0(\FF) - \ell(A/F_1) + \sum_{n \geq 2}^{r(\mathcal{F})} \ell(F_n/QF_{n-1}) {\mbox{\quad (as $ e_0(\FF) = \ell(A/Q))$}}
\end{eqnarray}
We observe that $\ell(F_n/QF_{n-1}) \geq 1$ for $2 \leq n \leq \nr(\mathcal{F})$. 
Therefore by \eqref{Eqn:e1}
\begin{equation} \label{Eqn:nr}
\nr(\FF)-1 \leq \sum_{n \geq 2}^{r(\mathcal{F})} \ell(F_n/QF_{n-1})  
= e_1(\FF) -e_0(\FF) + \ell(A/F_1),
\end{equation}
and hence $\nr(\FF) \leq e_1(\FF) -e_0(\FF) + \ell(A/F_1) +1$. 

\par \vspace{2mm}
(b) Suppose $\nr(\FF) = e_1(\FF) -e_0(\FF) + \ell(A/F_1) +1$.
Then by \eqref{Eqn:nr} $\ell(F_n/QF_{n-1}) = 0$ for 
$n\geq \nr_Q(\FF)+1 $ and $\ell(F_n/QF_{n-1})=1$ 
for $2 \leq n \leq \nr_Q(\FF)$. 
Therefore $\nr_Q(\FF) = r_Q(\FF)$, 
which yields $\nr(\FF)=r(\FF)$. 
The converse is also true. 
\par \vspace{2mm}
Again using \cite[Proposition 4.6]{HucMar} we have, for $i \geq 2$, \[
e_i(\FF) = \sum_{n \geq i}^{\nr(\FF)} \ell(F_n/QF_{n-1}) \binom{n-1}{i-1} 
= \sum_{n \geq i}^{\nr(\FF)} \binom{n-1}{i-1} = \binom{\nr(\FF)}{i}.
\]
\end{proof}

\par 
Using Theorem \ref{thm:nr(F)=r(F)} for the filtration $\FF = \{\ol{I^n}\}_{n\in \mathbb{Z}}$ in dimension two, we recover and extend the result \cite[Theorem 2.7]{ORWY2022} by Okuma, Rossi, Watanabe and Yoshida.

\begin{cor}\cite[Theorem 2.7]{ORWY2022} 
Let $(A, \m)$ be a two-dimensional Cohen-Macaulay analytically unramified local ring, and $I$ an $\m$-primary ideal.
Let  $Q$ a minimal reduction of $I$.
Then
\[
{\nr}(I)\leq \ol{e}_1(I)-e_0(I)+\ell(A/\ol{I})+1.
\]
Moreover, the equality holds if and only if 
$\ell(\ol{I^n}/Q\ol{I^{n-1}}) = 1$ for all $2 \leq n \leq {\nr}(I)$. 
In this case, ${\nr}(I) = \ol{r}(I)$. 
Also, $\ol{e}_2(I) = \binom{{\nr}(I)}{ 2}$. 
\end{cor}

\begin{proof}
Since $\depth{\overline{G}(I)} \geq 1$ 
by \cite[Proposition 3.25]{Mar89},  
the result follows from Theorem \ref{thm:nr(F)=r(F)}.    
\end{proof}

\begin{rem}
If $r(\FF)=1$, then $\nr(\FF)=r(\FF)$ is maximal. 
\par 
Suppose $d \ge 2$ and there exists a minimal reduction 
$Q$ of $\FF$ such that $r_Q(\FF)=2$ and $F_2 \cap Q=QF_1$. Then 
$\nr(\FF)$ is maximal if and only if $e_2(\FF)=1$.  
\end{rem}

\par 
Now we compute the Hilbert series of $G(\FF)$ when $\nr(\FF)$ is maximal.

\begin{thm}\label{thm:HSMaximal}
Let $(A, \m) $ be a Cohen-Macaulay local ring of dimension $d \ge 1$, and $\FF =  \{{F_n}\}_{n\in \mathbb{Z}}$ a Hilbert filtration. 
Suppose $\depth G(\FF) \geq d-1$, and $\nr(\FF)=r \geq 2$ is maximal. 
Put $\lambda=\ell(A/F_1)$ and $e_0=e_0(\FF)$. 
Then 
\begin{equation*}
HS_{{G}(\FF)}(t) = \frac{\lambda + (e_0-\lambda-1)t +t^r}{(1-t)^d}.
\end{equation*}
\end{thm}

\begin{proof}
First suppose $d=1$. 
Let $Q=(x)$  a minimal reduction of  $\FF$.  
By Theorem \ref{thm:nr(F)=r(F)} we have ${r}(\FF)=\nr(\FF)$ and $\ell({F_n}/Q{F_{n-1}})=1$ for all $2\leq n \leq {r}(\FF)$.
Observe that for $1\leq n \leq r-1$,  
         
\begin{eqnarray*}
\ell\left( {F_n}/{F_{n+1}} \right)&=&\ell(A/F_{n+1})-\ell(A/F_n) \nonumber\\
&=&\ell \left( A/x {F_n} \right)-\ell \left( {F_{n+1}}/x{F_n} \right) -
\ell \left( A/ {F_n} \right)   \nonumber\\
&=& \ell \left( A/xA \right)+\ell\left( A/{F_n} \right)
-\ell\left( {F_{n+1}}/x{F_n} \right)-\ell\left(A/{F_n}\right) \\ 
&=& e_0-1.  
\end{eqnarray*}

Therefore
\begin{align*}
HS_{{G}(\FF)}(t)
& =\sum_{n \ge 0}^{\infty}\ell({F_n}/{F_{n+1}})t^n \\
&=\lambda +\sum_{n= 1}^{r-1} (e_0-1) t^n
+\sum_{n\geq r} e_0 t^n \\[1mm]
&= \lambda + (e_0-1)\dfrac{t(1-t^{r-1})}{1-t} + \dfrac{e_0t^r}{1-t}  \\[1mm]
&=\dfrac{\lambda+(e_0-\lambda-1)t+t^r}{1-t}. 
\end{align*}

\par \vspace{2mm}
Next suppose $d \ge 2$. 
We take a sequence $x_1,x_2,\ldots,x_d$ in $F_1$ which satisfies 
the condition of Lemma \ref{Reduce}.  
Put $Q=(x_1,x_2,\ldots,x_d)A$, $\underline{x}=x_1,\ldots,x_{d-1}$, 
$\overline{A}=A/(\underline{x})A$, and $\overline{\FF}=\FF/(\underline{x})$. 
Then  $\nr(\overline{\FF})$ is also a maximal reduction number. 
Hence it follows from the argument as above that 
\[
HS_{G(\overline{F})}(t)=\dfrac{\lambda+(e_0-\lambda-1)t+t^r}{1-t} 
\]
and thus 
\[
HS_{G(\FF)}(t) =\dfrac{HS_{G(\overline{F})}(t)}{(1-t)^{d-1}}=
\dfrac{\lambda+(e_0-\lambda-1)t+t^r}{(1-t)^d}.  
\]
\end{proof}

\par 
As a consequence, we provide criteria for $G(\FF)$ to be Gorenstein 
in terms of Hilbert coefficients in the case where $\nr(\FF)$ is maximal. 

\begin{thm}\label{thm:GorMaximal}
Let $(A, \m) $ be a Gorenstein local ring of dimension $d \ge 1$,  
and $\FF =  \{{F_n}\}_{n\in \mathbb{Z}}$ a Hilbert filtration. 
Suppose that $\nr(\FF) =r \geq 2$ is maximal and 
$G(\FF)$ is Cohen-Macaulay. 
Then  
\begin{enumerate}
\item[\rm{(a)}] When $r =2$, 
$G(\FF)$ is Gorenstein if and only if $F_1 = \m$. 
\item[\rm{(b)}] When $r \ge 3$, 
$G(\FF)$ is Gorenstein if and only if $F_1 = \m$  and $e_0(\m) = 2$.
\end{enumerate}
\end{thm}

\begin{proof}
(a) By Theorem \ref{thm:HSMaximal}, 
we have 
\[
HS_{G(\FF)}(t)=\dfrac{\lambda+(e_0-\lambda-1)t+t^2}{(1-t)^d}.
\]    
Thus, by Theorem \ref{thm:GoriffhVectorCondition}, 
$G(\FF)$ is Gorenstein if and only if  $\lambda=1$, that is, 
$F_1=\m$. 
\par \vspace{2mm}
(b) By Theorem \ref{thm:HSMaximal}, 
we have 
\[
HS_{G(\FF)}(t)=\dfrac{\lambda+(e_0-\lambda-1)t+t^r}{(1-t)^d}.
\]    
Thus, by Theorem \ref{thm:GoriffhVectorCondition},  
$G(\FF)$ is Gorenstein if and only if  $\lambda=1$ and $e_0-\lambda-1=0$, that is, 
$F_1=\m$ and $e_0(\m)=2$.
\end{proof}

Applying Theorem \ref{thm:GorMaximal} 
for the filtration $\{\ol{I^n}\}_{n \in \mathbb{Z}}$ 
we get the following result.

\begin{cor} \label{cor:Maximal_normal}
Let $(A, \m) $ be a analytically unramified Gorenstein local ring of dimension $d \ge 1$, and $I$ an $\m$-primary ideal in $A$. 
Suppose that $r=\br(I)=\nr(I)$ is maximal and $\olsi{G}(I)$ is Cohen-Macaulay. 
\begin{enumerate}
\item[\rm{(a)}] When $\br(I)=2$,  
$\olsi{G}(I)$ is Gorenstein if and only if $\ol{I} = \m$. 
\item[\rm{(b)}] When $\br(I) \ge 3$,  
$\olsi{G}(I)$ is Gorenstein if and only if $\ol{I} = \m$ and ${e_0}(\m) = 2$.
\end{enumerate}
\end{cor}

\begin{proof}
    Follows from Theorem \ref{thm:GorMaximal}.
\end{proof}

\section{Examples} \label{sec:Examples}

Note that if $(A, \m)$ is a positive dimensional Cohen-Macaulay local ring and $I$ is an $\m$-primary ideal, then $\depth \ol{G}(I) \geq 1$ by \cite[Proposition 3.25]{Mar89}. In particular, if $\dim A = 1$, then $\ol{G}(I)$  is Cohen-Macaulay.

\par 
Next, we give examples of one-dimensional semigroup rings and apply our results.

\begin{defn} \label{NormalRedSemiG}
Let $m_1 < m_2< \cdots< m_n$ be positive integers with 
$\gcd(m_1, m_2, \ldots, m_n) = 1$, 
and $S$ be the semigroup generated by $m_1, m_2, \ldots, m_n$. 
Then the Fr\"obenius number $F(S)$  of $S$ is defined as the 
smallest integer $k$ such that $n \in S$ for all $n >k$.
\par 
A semigroup $S$ is \textit{symmetric} if $F(S)$ is odd and for every $a \in \bbZ$, 
either $a \in S$ or $F(S)-a \in S$. 
\end{defn}

\begin{lem} \label{lemma:redNumSemG}
Let $A = K\llbracket t^{m_1}, t^{m_2}, \ldots, t^{m_n} \rrbracket$ 
be a semigroup ring of the semigroup generated 
by $m_1 < m_2< \cdots< m_n$ with $\gcd(m_1, m_2, \ldots, m_n) = 1$. 
Then
\[
\ol{r}(\m) = \left\lceil \frac{F(S)}{m_1} \right\rceil.
\] 
\end{lem}
 
\begin{proof}
We note that $Q =  (t^{m_1}) $ is a minimal reduction of 
$\m = (t^{m_1}, t^{m_2}, \ldots, t^{m_n})$, 
the maximal ideal of $A$. 
By \cite[Example 7.3]{HKU} we have $\ol{\m^i} = \{\alpha \in A \mid \textrm{ord}(\alpha) \geq m_1i \}$ where $\textrm{ord}(\alpha)$ is the degree of the initial form of $\alpha$. 
Indeed, the natural inclusion map 
$A \to K\llbracket t \rrbracket$ is an integral extension. 
We have $\ol{I} = \ol{IK\llbracket t \rrbracket} \cap A$ 
for any ideal $I$ in $A$. Since $K\llbracket t \rrbracket$ is a PID, $(t^{m_1})K\llbracket t \rrbracket$ is integrally closed. 
As $Q = (t^m_1)$ is a minimal reduction of $\m$, 
$\ol{Q^i} = \ol{\m^i}.$ Let $k$ be the smallest integer such that $m_1k>F(S).$
\par     
Consider $\ol{\m^k} =  (t^{m_1k},  t^{m_1k+1}, \cdots )$. 
Hence 
\[
Q\ol{\m^k} = (t^{m_1(k+1)}, t^{m_1(k+1)+1} \ldots ) = \ol{\m^{k+1}}.
\] 
Therefore $\ol{r}(\m) \leq k$. 
\par 
By the definition of the Fr\"obenius number $m_1(k-1) \neq F(S)$. 
Hence, $m_1(k-1) < F(S)$. 
Hence $F(S)+ m_1 > m_1k$. 
Therefore $t^{F(S)+ m_1} \in \ol{\m^{k}}$. 
Suppose $t^{F(S)+ m_1} \in Q\ol{\m^{k-1}}$. 
Then $t^{F(S)+ m_1} = t^{m_1}b$ where $b \in A$. 
Since $t^{m_1}$ is a regular element in $k\llbracket t \rrbracket$, 
we obtain $t^{F(S)} = b \in A$. 
This is a contradiction. 
Hence $t^{F(S)+ m_1} \in \ol{m^{k}} \setminus Q\ol{\m^{k-1}}$. Therefore $\ol{r}(\m) = k$.
\end{proof}

\begin{exam}\label{Exam:3and4}
Let $A =K\llbracket t^3,t^4 \rrbracket$ be a semigroup ring generated by $t^3$ and $t^4$. 
The Fr\"obenius number $F(S)$ of the semigroup 
$S = \BN 3 +\BN 4$ is $5$. 
Since $S$ is symmetric, $A$ is Gorenstein with $e_0(\m)=3$. 
Also, $\ol{\m^i} = \{\alpha \in A \mid \textrm{ord}(\alpha) \geq 3i \}$.  Note that $Q = (t^3)$ is a minimal reduction of $\m$. 
Hence by Lemma \ref{lemma:redNumSemG}, $\ol{r}(\m) = 2$. 
Using Equation (\ref{redandpos}), $\ol{n}(\m) = 2-1 =1$. 
Hence $\ol{P}_{\m}(n) = \ol{H}_{\m}(n)$ for all $n \geq 2$. 
Therefore $\ol{e_1}(m)=2e_0(\m)-\ell(A/\ol{\m^2})=2 \cdot 3 -3 =3$. 
Since $e_0(\m) = \ol{e_1}(\m)$, by Corollary \ref{cor:red12_normal}, 
$\ols{G}(\m)$ is Gorenstein.
\end{exam}   

\begin{exam}\label{Exam:2andm}
Let $A =K\llbracket t^2,t^m \rrbracket$ be a semigroup ring 
generated by $t^2$ and $t^m$, where $m$ is an odd integer $\geq 3$. The Fr\"obenius number $F(S)$ of the semigroup 
$S = \BN 2 +\BN m$ is $m-2$. Since $S$ is symmetric, $A$ is Gorenstein. 
Also, $\ol{\m^i} = \{\alpha \in A \mid \textrm{ord}(\alpha) \geq 2i \}$.  Note that $Q = ( t^2 )$ is a minimal reduction of $\m$. 
Hence by Lemma \ref{lemma:redNumSemG}, $\ol{r}(\m) = (m-1)/2$.
Therefore $\ol{P}_{\m}(n) = \ol{H}_{\m}(n)$ for all $n \geq (m-1)/2$. 
Also, since $e_0(\m)=2$, we have 
\[
\ol{e_1}(\m) = e_0(\m)((m-1)/2)-\ell(A/\ol{\m^{(m-1)/2}})
=(m-1)-(m-1)/2=(m-1)/2. 
\]
Hence $\ol{r}(\m)= \ol{e_1}(\m) -e_0(\m) +\ell(A/\ol{m})+1$ and thus 
$\nr(\m)$ is maximal. 
Therefore, by Corollary \ref{cor:Maximal_normal}, $\olsi{G}(\m)$ is Gorenstein.
\end{exam}  

\begin{exam} \label{ex:456}
Let $A =K\llbracket t^4,t^5, t^6 \rrbracket$ be a semigroup ring 
generated by $4,5$ and $6$. 
The Fr\"obenius number of the semigroup $F(S) = 7$. 
Since $S$ is symmetric, $A$ is Gorenstein. 
By Lemma \ref{lemma:redNumSemG} we get $\ol{r}(\m) = 2$. 
Using similar arguments as above we obtain $e_0(\m) = 4$ 
and $\ol{e_1}(\m) = 4$. 
Hence, by Corollary \ref{cor:red12_normal},  
$\olsi{G}(\m)$ is Gorenstein.
\end{exam}    

\par 
The following example is from \cite[Example 7.3]{HKU}. 
They prove that $\ols{G}(\m)$ is not Gorenstein. 
We also prove the same by using our techniques.

\begin{exam} \label{ex:467}
Let $A =K\llbracket t^4,t^6, t^7 \rrbracket$ be a semigroup ring 
generated by $4,6$ and $7$. 
The Fr\"obenius number of the semigroup $F(S) = 9$ and 
the semigroup is symmetric. 
Hence $A$ is a one-dimensional Gorenstein local ring. 
Let $\m = (t^4, t^6, t^7)$. 
Then $\ol{\m^i} = \{\alpha \in A \mid \textrm{ord}(\alpha) \geq 4i \}$, 
and $Q = (t^4)$ is a minimal reduction of $\m$ 
with reduction number $\ol{r}(\m) = 3$ by Lemma \ref{lemma:redNumSemG}.  
By Equation (\ref{redandpos}), $\ol{n}(\m) = 3 -1 =2$. 
Hence $\ol{P}_{\m}(n) = \ol{H}_{\m}(n)$ for all $n \geq 3$. 
In particular, 
$\ol{e_1}(\m) = 3e_0(\m)-\ell(A/ \ol{\m^3})=3 \cdot 4-7=5$. 
Then since $\nr(\m)=3=\ol{e_1}(\m) - e_0(\m) + \ell(A/\m)+ 1$ is maximal 
but $e_0(\m) = 4$, by Theorem \ref{thm:GorMaximal}, 
$\ols{G}(\m)$ is {\it not} Gorenstein.
\end{exam}    

In the rest of this section, let $(A,\m)$ be a $2$-dimensional 
excellent normal local domain with algebraically closed residue field, and 
let $I \subset A$ be an $\m$-primary integrally closed ideal.  

\begin{exam}
Let $A$ and $I$ be as above. 
If $I$ is a {\it $p_g$-ideal} (i.e. $\ol{r}(\m)=1$), then $\nr(\m)$ is maximal. 
\par 
Also, if $I$ is an {\it elliptic ideal} (i.e. $\ol{r}(\m)=2$), 
then $\nr(\m)$ is maximal if and only if $I$ is a 
{\it strongly elliptic ideal} (i.e. $\ol{e_2}(\m)=1$). 
See \cite{OWY2016, ORWY2022} for details. 
\end{exam}

\begin{exam}
Suppose that $I$ is a $p_g$-ideal. 
Then $\ols{G}(I)$ is Gorenstein if and only if $I$ is {\it good}, that is, 
$I^2=QI$ and $I=Q\colon I$ for some minimal reduction $Q$ of $I$. 
\par 
On the other hand, for $A=\mathbb{C}[[x,y,z]]/(x^3+y^3+z^3)$ and $\m=(x,y,z)A$,  we have that $\ol{r}(\m)=2$ and 
$\ols{G}(\m)$ is Gorenstein but $\m$ is {\it not} 
a good ideal.  
\end{exam}

\section{Gorensteinness of $\ol{G}(\m)$ in Zariski hypersurface rings}\label{sec:GorGm}
 
\par 
Consider the following setup in the rest of the paper. 
Let $K$ be an algebraically closed field, and 
let $m \ge 1$, $2 \le a \le b$ be integers, 
and put $S=K[[x, y_1,\ldots, y_m]]$ be a formal power series ring 
with $(m+1)$-variables over $K$. 
Then $S$ is a complete regular local ring with 
the unique maximal ideal $\m_S =(x, y_1,\ldots,y_m)$. 
Put $\underline{y}=y_1,\ldots,y_m$.
Let $g$, $f$  be polynomials such that  
\begin{eqnarray*}
g &= & g(\underline{y}) \in (\underline{y})^b 
\setminus (\underline{y})^{b+1}, \\
f &=& f(x,\underline{y})=x^a-g(\underline{y}).
\end{eqnarray*}
Then a hypersurface $A=S/(f)$  
is called \lq\lq  {\it Zariski-type hypersurface}'' if $a \ne 0$ in $K$. 
Let $Q=(\underline{y})A$.   
Note that $Q$ is a minimal reduction of $\m=\m_SA$. 
Indeed, $\m^a=Q\m^{a-1}$. 
In this section, we compute $\br(\m)$ and prove the Cohen-Macaulayness of $\overline{G}(\m)$. 
In order to do that, we generalize several results for Brieskorn hypersurfaces in \cite[Section 3]{OWY2019} to that of Zariski-type hypersurfaces. 
We fix the following notation for the rest of the paper:
\[
d=\gcd(a,b), \quad a'=\frac{a}{d}, \quad b'=\frac{b}{d}
\quad 
\text{and} 
\quad 
n_k=\lfloor \frac{kb}{a} \rfloor \quad \text{for $k=1,2,\ldots,a-1$.}
\]
\par 
For all $n \ge 1$, we define the following two ideals:
\begin{eqnarray*}
J_n &:=& Q^n+xQ^{n-n_1}+x^2Q^{n-n_2} + \cdots + x^{a-1}
Q^{n-n_{a-1}}, \\
I_n &:=& \left(x^k y_1^{i_1}\cdots y_m^{i_m} \,\big|\,kb'+(i_1+\cdots+i_m)a' \ge n, k,i_1, \ldots, i_m \in \mathbb{Z}_{\geq 0} \right)A.
\end{eqnarray*}

\begin{prop} \label{JNvsIN}
$J_n=I_{na'}$ for every $n \ge 1$. 
\end{prop}

\begin{proof}
First we prove $J_n \subset I_{na'}$. 
It is enough to show $x^kQ^{n-n_k} \subset I_{na'}$ for every $k=1,2,\ldots,a-1$. 
For each $k$ with $1 \le k \le a-1$, if $i_1+\cdots +i_m \ge n-n_k$, 
then 
\[
kb'+(i_1+\cdots +i_m)a' 
 \ge  kb'+(n-n_k)a' 
 \ge  kb'+\big(n-\frac{kb'}{a'}\big)a'=na',  
\]
as required. Next we prove the converse. 
If $x^ky_1^{i_1}\cdots y_m^{i_m} \in I_{na'}$, then 
$kb'+(i_1+\cdots+i_m)a' \ge na'$. 
Thus
\[
n_k=\lfloor \frac{kb'}{a'} \rfloor \ge n-(i_1+\cdots+i_m), 
\;\text{that is,}\; 
i_1+\cdots+i_m \ge n-n_k. 
\]
Hence $x^ky_1^{i_1}\cdots y_m^{i_m} \in x^k Q^{n-n_k}$. 
This yields $I_{na'} \subset J_n$. 
\end{proof}

\begin{prop} \label{JN}
Under the notation above, 
\begin{enumerate}
\item[\rm{(a)}] For every $k,n \ge 1$, $x^k \in \overline{Q^n}$ 
if and only if $n \le n_k$. 
\item[\rm{(b)}] 
$Q^n \subset J_n \subset \overline{Q^n}$, that is, 
$\overline{J_n}=\overline{Q^n}$ for each $n \ge 1$. 
\end{enumerate}
\end{prop}

\begin{proof}
(a) We first show that if $n \le n_k$ then $x^k \in \overline{Q^n}$. 
By definition of $n_k$, we get 
\[
(x^k)^a=x^{ka} \in (Q^b)^k =Q^{kb} \subset Q^{an_k}=(Q^{n_k})^a.  
\]
Hence $x^k \in \overline{Q^{n_k}} \subset \overline{Q^{n}}$. 
\par \vspace{2mm}
Suppose $x^k \in \overline{Q^{n}}$. 
We will show $n \le n_k$. 
By assumption, there exists a nonzero divisor $c \in A$ such that 
$c(x^k)^{\ell} \subset Q^{n\ell}$ for sufficiently large $\ell$. 
By Artin-Rees' lemma, we can take an integer $\ell_0 \ge 1$ such that 
if $\ell \gg 0$, then $Q^{\ell} \cap cA =c(Q^{\ell} \colon c) \subset cQ^{\ell-\ell_0}$. 
Thus $Q^{\ell} \colon c \subset Q^{\ell-\ell_0}$ 
because $c$ is a  nonzero divisor.
\par 
Now suppose that $n \ge n_k+1$. 
By the choice of $n_k$, we have 
\[
\frac{kb}{a}+\frac{1}{a} \le n_k+1 \le n.
\]
Then $kb+1 \le a(n_k+1) \le an$. 
If $\ell \ge \ell_0+1$, then 
\[
g^{k\ell} =x^{ak\ell} \in Q^{an\ell}\colon c \subset Q^{an\ell-\ell_0}
\subset Q^{(kb+1)\ell-\ell_0}
\subset Q^{kb\ell+1}. 
\]
Moreover, since $g \in Q^b \setminus Q^{b+1}$, if necessary, by replacing $\underline{y}=y_1,y_2,\ldots,y_m$ 
with the other minimal generators of $Q$, 
we may assume that $g \equiv y_1^{b} \pmod{(y_2,\ldots,y_m)A}$.  
Hence $y_1^{bk\ell} \in (y_1^{kb\ell+1},y_2,\ldots, y_m)$. 
This contradicts the fact that $y_1,y_2,\ldots,y_m$ forms a regular sequence on $A$. 
\par \vspace{2mm} 
(b)
It is enough to show 
$x^ky_1^{i_1}\cdots y_m^{i_m} \in \overline{Q^n}$ 
if and only if $i_1+\cdots+i_m \ge n-n_k$. 
Since $\overline{G}(Q) \cong \overline{\mathcal{R}'}(Q)/t^{-1}\overline{\mathcal{R}'}(Q)$ and 
$\overline{\mathcal{R}'}(Q)$ is normal, we have $\depth \overline{G}(Q) \ge 1$. 
Hence we obtain $\overline{Q^n} \colon y_j = \overline{Q^{n-1}}$ for every $n,j$. 
So one can easily see that 
\[
x^ky_1^{i_1}\cdots y_m^{i_m} \in \overline{Q^n} \Longleftrightarrow 
x^k \in \overline{Q^{n-(i_1+\cdots+i_m)}}. 
\]
By (a), the last condition means that 
$i_1+ \cdots +i_m \ge n-n_k$.  Hence we get the required inclusions. 
\end{proof}

\par \vspace{2mm}
Set $\mathcal{I}=\{I_{n}\}_{n \in \bbZ}$, which is a filtration of ideals in $A$ introduced above and, $I_{n}=A$ for every integer $n \le 0$. 
We put
\[
\mathcal{R}'(\mathcal{I})= \bigoplus_{n \in \bbZ} I_{n} \,t^{n}  \subset A[t,t^{-1}], \qquad 
G(\mathcal{I})  =\bigoplus_{n \ge 0} I_{n}/I_{n+1} 
\cong \mathcal{R}'(\mathcal{I})/t^{-1}\mathcal{R}'(\mathcal{I}). 
\]

In the following, we set $\underline{y}t^k = y_1t^k,\ldots,y_mt^k$ and 
$\underline{Y}=Y_1,\ldots,Y_m$. 

\begin{lem} \label{IN-filt}
Write $g=g_b+g_{b+1}+\cdots+g_c$, where $g_i$ denotes a homogeneous polynomial of degree $i=b,b+1,\ldots,c$. 
Then 
$\mathcal{R}'(\mathcal{I}) \cong A[xt^{b'}, \underline{y}t^{a'}, t^{-1}]$  and $ G(\mathcal{I})  \cong K[X, \underline{Y}]/(X^a-g_b(\underline{Y})).$
\end{lem} 

\begin{proof}
We have 
\begin{equation*}
 \begin{split}
  \mathcal{R}'(\mathcal{I}) 
  & \cong \bigoplus_{n \in \bbZ} I_{n} \,t^{n} \\
  & \cong \sum_{k, i_1, \ldots, i_m \in \mathbb{Z}_{\geq 0}} Kx^k 
  y_1^{i_1}\cdots y_m^{i_m} \, t^{kb' +(i_1+\cdots+i_m)a'} +\sum_{n \geq 0}A \,t^{-n} \\
  & = \sum_{k, i_1, \ldots, i_m \in \mathbb{Z}_{\geq 0}}  K(xt^{b'})^k (y_1t^{a'})^{i_1} \cdots (y_mt^{a'})^{i_m} +\sum_{n \geq 0}A \, t^{-n} \\
  & = A[xt^{b'}, y_1t^{a'}, \ldots, y_mt^{a'}, t^{-1}].
  \end{split}
\end{equation*}
Put $X = xt^{b'}, Y_i = y_it^{a'}$ for $i = 1, 2, \ldots, m$  and $U = t^{-1}$. Then
\begin{equation*}
 \begin{split}
  X^a & = x^at^{ab'} = x^a t^{a'b} \\
        & = g_b(y_1, \ldots, y_m)t^{a'b} +g_{b+1}(y_1, \ldots, y_m)t^{a'b}+ \cdots + g_c(y_1, \ldots, y_m)t^{a'b} \\
        & = g_b(Y_1, Y_2, \ldots, Y_m) +g_{b+1}(Y_1, Y_2, \ldots, Y_m)t^{-a'} + \cdots + g_c(Y_1, Y_2, \ldots, Y_m)t^{-(c-b)a'} \\
        & = g_b(Y_1, Y_2, \ldots, Y_m) +g_{b+1}(Y_1, Y_2, \ldots, Y_m)U^{a'} + \cdots + g_c(Y_1, Y_2, \ldots, Y_m)U^{(c-b)a'}.
    \end{split}
\end{equation*}
Hence 
\begin{equation*}
    \begin{split}
        G(\mathcal{I}) & \cong \mathcal{R}'(\mathcal{I})/t^{-1}\mathcal{R}'(\mathcal{I}) \\
        & \cong  k[X, Y_1, Y_2, \ldots, Y_m]/(X^a-g_b(\underline{Y})).
    \end{split}
\end{equation*}
\end{proof}

\par \vspace{2mm}
The normality of $A$ is needed in the following proposition. 

\begin{prop} \label{IntClos}
Suppose that $A$ is a Zariski-type hypersurface. 
Under the notation as above, 
\begin{enumerate}
\item[\rm{(a)}] $G(\mathcal{I})$ is reduced. 
\item[\rm{(b)}] $\overline{I_{n}}=I_{n}$ for every $n \ge 1$. 
\item[\rm{(c)}] $\overline{J_{n}}=J_{n}$ for every $n \ge 1$. 
\end{enumerate}
\end{prop}

\begin{proof}
(a) By Lemma \ref{IN-filt}, since $G(\mathcal{I}) \cong K[X, \underline{Y}]/(X^a-g_b(\underline{Y}))$ is a hypersurface, it satisfies 
the Serre condition $(S_1)$. 

\par \vspace{1mm}
Put $f=X^a-g_b(\underline{Y})$. 
Then 
\[
\height \left(\frac{\partial f}{\partial X}, 
\frac{\partial f}{\partial Y_1}, \ldots, \frac{\partial f}{\partial Y_m}, f \right) \ge \height (X, g_b(\underbar{Y})) \ge 1.
\] 
Hence $G(\mathcal{I})$ satisfies $(R_0)$ and it is reduced. 
\par \vspace{2mm} \par \noindent 
(b) It follows from \cite[Lemma 3.1]{MMV12}.

\par \vspace{2mm}\par \noindent 
 (c) Since $J_n=I_{n\alpha'}$ for every $n \ge 1$ by Proposition \ref{JNvsIN},  
$J_n$ is also integrally closed by (b). 
 \end{proof}

\par \vspace{2mm}
The following theorem generalizes \cite[Theorem 3.1]{OWY2019}. 

\begin{thm} \label{S1-Main}
Under the notation as above, 
\begin{enumerate}
\item[\rm{(a)}] $\overline{\m^{n}}=\overline{Q^n}=Q^n
+xQ^{n-n_1}+x^2Q^{n-n_2}+\cdots+x^{a-1}Q^{n-n_{a-1}}$. 
\item[\rm{(b)}] $\br_Q(\m)=\nr_Q(\m)=n_{a-1}= \lfloor \frac{(a-1)b}{a}\rfloor$. 
\item[\rm{(c)}] $\overline{G}(\m)$ is Cohen-Macaulay. 
\end{enumerate}
\end{thm}

\begin{proof}
(a) It follows from Proposition \ref{JN}(b) and Proposition \ref{IntClos}(c). 
\par \vspace{2mm} \par \noindent 
(b) It immediately follows from (a). 
\par \vspace{2mm} \par \noindent 
(c) By Proposition \ref{IntClos}, $G(\mathcal{I})$ is a hypersurface. Therefore it is Cohen-Macaulay. 
Hence $\mathcal{R}'(\mathcal{I})$ is Cohen-Macaulay. 
\par \vspace{2mm}
On the other hand, 
$\overline{\mathcal{R}'}(\m)=\mathcal{R}'(\{J_n\})= 
\mathcal{R}'(\{I_{na'}\})$ by  Propositions \ref{JNvsIN}, \ref{JN}(b) and \ref{IntClos}(c). 
Then $\overline{\mathcal{R}'}(\m)$  is also Cohen-Macaulay since $\overline{\mathcal{R}'}(\m)$ is isomorphic to the $a'$th Veronese 
subring of $\mathcal{R}'(\mathcal{I})$. 
Therefore $\overline{G}(\m)$ is Cohen-Macaulay. 
\end{proof}

\bigskip 
\par  \vspace{2mm}
Next, we give a criterion for $\overline{G}(\m)$ to be 
Gorenstein in the case $A$ is a Zariski-type hypersurface. 
First, we need the following criterion from \cite{OWY7}. 
We put $L_n=Q+\overline{\m^{n}}$ and $\ell_n=\ell(A/L_n)$ 
for every $n \ge 1$. 

\begin{lem}[\textrm{cf. \cite[Proposition 2.5]{OWY7}}] \label{MainLem}
Put $r=r_Q(\m)$. Then the following conditions are equivalent$:$
\begin{enumerate}
\item[\rm{(a)}] $\overline{G}(\m)$ is Gorenstein. 
\item[\rm{(b)}] 
$Q \colon L_n=L_{r+1-n}$ for every $n=1,2,\ldots, r$. 
\item[\rm{(c)}]
$\ell_n+\ell_{r+1-n}=a$ for  every $n=1,2,\ldots,\lceil \frac{r}{2} \rceil$. 
\end{enumerate}
\end{lem}

As an application of Proposition \ref{IntClos}, we obtain the following proposition.

\begin{prop} \label{LNlength}
Under the same notation as above, we have 
\[
\ell_n=\min\{k \in \bbZ_{+} \,|\,  n \le n_k\}.
\] 
\end{prop}

\begin{proof}
Suppose $n > n_{k-1}$ and $n \le n_k$. 
Then 
\begin{eqnarray*}
L_n & =& Q+\overline{\m^n} \\
&=& Q+(Q^n+xQ^{n-n_1}+\cdots + x^{k-1}Q^{n-n_{k-1}} 
+x^kQ^{n-n_k}+\cdots ) \\
&=& Q+(x^k). 
\end{eqnarray*}
Hence $\ell_n=\ell(K[x]/(x^k))=k$. 
\end{proof}

\par \vspace{2mm}
The following theorem is a main result in this section. 
Recall $d=\gcd(a,b)$. 

\begin{thm} \label{Gor-Main}
$\overline{G}(\m)$ is Gorenstein if and only if $b \equiv 0,d \pmod{a}$. Moreover, 
\begin{enumerate}
\item[\rm{(a)}] If $b \equiv 0 \pmod{a}$, 
then $\overline{G}(\m)$ is a reduced, hypersurface. 
\item[\rm{(b)}] If $b \equiv d \pmod{a}$ and $d=1$,  
then $\overline{G}(\m)$ is a non-reduced, hypersurface. 
\item[\rm{(c)}] If $b \equiv d \pmod{a}$ and $d>1$, 
then $\overline{G}(\m)$ is a complete 
intersection but not a hypersurface. 
\end{enumerate}
\end{thm}

\par \vspace{2mm}
In what follows, we prove the above theorem in several steps.
By Theorem \ref{S1-Main}, we have 
\begin{eqnarray*}
\m^n t^n &=& (Q^n+ xQ^{n-n_1} + \cdots + x^{a-1}Q^{n-n_{a-1}})t^n \\[1mm]
&=& (Qt)^n+ (xt^{n_1})(Qt)^{n-n_1} + \cdots 
+ (x^{a-1}t^{n_{a-1}})(QT)^{n-n_{a-1}}. 
\end{eqnarray*}
Thus 
\[
\overline{\mathcal{R}'}(\m) = K[xt,\underline{y}t,\{x^kt^{n_k}\}_{1 \le k \le a-1},t^{-1}]. 
\]
\par 
Using this fact, we first consider the case where $a$ divides $b$.  

\begin{prop} \label{daNGT}
If $a$ divides $b$, then $\overline{G}(\m)$ is a reduced hypersurface, and thus Gorenstein. In fact, 
\[
\overline{G}(\m)  \cong 
K[X, \underline{Y}]/(X^a - g_b(\underline{Y})). 
\]
\end{prop}

\begin{proof}
As $b=b'a$ for some integer $b' \ge 1$, we have $n_k=\lfloor kb' \rfloor =kb'$ for every $k=1,2,\ldots,a-1$.  
Hence one can easily see that 
\begin{eqnarray*}
\overline{\mathcal{R}'}(\m) 
&=& K[xt, \underline{y}t, \{x^kt^{n_k}\}_{1 \le k \le a-1}, t^{-1}] \\
&=& K[xt^{n_1}, \underline{y}t, \,t^{-1}] \\
& \cong &  K[X, \underline{Y},U]/
(X^a-g_b(\underline{Y})-\cdots-g_c(\underline{Y})U^{c-b}),  
\end{eqnarray*}
where $X=xt^{n_1}$, $Y_i=y_it$ $(i=1,2,\ldots,m)$, and $U=t^{-1}$. 
\par \vspace{2mm} 
Moreover, the claims on $G(\m)$ 
immediately follow from the fact that 
\[
\overline{G}(\m) \cong \overline{\mathcal{R}'}(\m)/t^{-1}
\overline{\mathcal{R}'}(\m).
\] 
\end{proof}

Secondly we consider the case of $d=1$ and $b \equiv 1 \pmod{a}$. 

\begin{prop} \label{dequiv1PC}
If $d=1$ and $b=n_1a+1$, then 
$\overline{G}(\m)$ is a non-reduced hypersurface, and thus Gorenstein. 
In fact, 
\[
\overline{G}(\m) 
 \cong  K[X,\underline{Y}]/(X^a) \cong G(\m).  
\]
\end{prop}

\begin{proof}
For every $k$ with $1 \le k \le a-1$, we have 
\[
n_k
=\lfloor \frac{kb}{a} \rfloor 
= \lfloor \frac{k(n_1a+1)}{a} \rfloor
=kn_1+ \lfloor \dfrac{k}{a} \rfloor = kn_1. 
\]
Hence 
\begin{eqnarray*}
\overline{\mathcal{R}'}(\m) &=& 
 K[xt, \underline{y}t,\{x^kt^{n_k}\}_{1 \le k \le a-1}, t^{-1}] \\
&=&  K[xt, \underline{y}t, xt^{n_1},t^{-1}] \\
&\cong & K[X,\underline{Y},U]/(X^a-g_b(\underline{Y})U-
\cdots -g_c(\underline{Y})U^{c-b+1}). 
\end{eqnarray*}
Indeed, put  $X = xt^{n_1}, Y_i = y_it$ for $i = 1, 2, \ldots, m$  and $U = t^{-1}$. Then $xt=(xt^{n_1})(t^{-1})^{n_1-1}$ and 
\begin{eqnarray*}
X^a=(xt^{n_1})^a &=& (g_b+g_{b+1}+\cdots+g_c)t^{an_1} \\
&=& (g_b+g_{b+1}+\cdots+g_c)t^{b-1} \\
&=& (g_bt^b)U+g_{b+1}t^{b+1}U^2+\cdots+g_ct^cU^{c-b+1} \\
&=& g_b(\underline{Y})U+\cdots +g_c(\underline{Y})U^{c-b+1}. 
\end{eqnarray*}
The other assertion follows from this. 
\end{proof}

\par \vspace{2mm}
Thirdly, we consider the case $1 <  d < a$, $b \equiv d \pmod{a}$, that is, $b' \equiv 1 \pmod{a'}$. 
\begin{prop} 
If $1<d<a$ and $b'=n_1a'+1$, then  
$\overline{G}(\m)$ is a complete intersection, and thus Gorenstein. 
In fact, 
\[
\overline{G}(m) \cong K[X,\underline{Y},Z]/(X^{a'}, Z^d-g_b(\underline{Y})).
\] 
\end{prop}

\begin{proof}
Note that $b'\equiv 1 \pmod{a'}$. 
One can easily see that $n_k=kn_1$ for $1 \le k \le a'-1$, 
and $n_{a'}=b'$. 
For any integer $p$ $(1\le p \le d-1)$, if $pa'+1 \le k \le pa'+a'-1$, 
then 
\[
n_{k} =\lfloor \frac{k(n_1a'+1)}{a'} \rfloor 
= kn_1+ \lfloor \frac{k}{a'} \rfloor 
= kn_1+p = pb'+(k-pa')n_1. 
\]
Hence $(k,n_k)=p(a',b')+(k-pa')(1,n_1)$. 
Similarly, if $k=pa'$ $(2 \le p \le d-1)$, then $(k,n_k)=p(a',b')$. 
Therefore
\begin{eqnarray*}
\overline{\mathcal{R}'}(\m) 
&=& K[xt,y_1t,\ldots,y_mt,xt^{n_1},x^{a'}t^{b'},t^{-1}] \\[1mm]
& \cong & K[X,\underline{Y},Z,U]/(X^{a'}-ZU, Z^d-g_b({\underline{Y}})- \cdots -
g_c({\underline{Y}})U^{c-b}), 
\end{eqnarray*} 
where $X=xt^{n_1}$, $Y_i=y_it$, $Z=x^{a'}t^{b'}$, and $U=t^{-1}$. 
Indeed, 
\[
X^{a'}=x^{a'}t^{n_1a'}=(x^{a'}t^{b'})(t^{-1})^{b'-n_1a'}=(x^{a'}t^{b'})(t^{-1})=ZU.
\]
and 
\begin{eqnarray*}
Z^d &=&(x^{a'}t^{b'})^d \\
&=& x^at^b \\
&=&
g_b(y_1,\ldots,y_m)t^b+\cdots+g_c(y_1,\ldots,y_m)t^c \cdot t^{b-c}  \\
&=& g_b(\underline{Y})+\cdots + g_c(\underline{Y})U^{c-b}. 
\end{eqnarray*}
The other assertion follows from this. 
\end{proof}

\par \vspace{2mm}
In the rest of this section, we consider the case of $1 < d < a$ 
and $b' \equiv \delta \pmod{a'}$ for some $\delta$ 
$(2 \le \delta \le a'-1)$. 
Note that 
\[
\br(\m) =n_{a-1}=\lfloor \dfrac{(a-1)b'}{a'} \rfloor = (a-1)n_1+d\delta-1. 
\]
In order to prove Theorem \ref{Gor-Main}, it suffices to prove the 
following proposition. 

\begin{prop}
In the above notation, if we put $k_1=\lceil \frac{a'}{\delta} \rceil$ and 
$t_1=(k_1-1)n_1+1$, then 
\[
\ell_{t_1}+\ell_{r+1-t_1} =a+1 (\ne a). 
\]
In particular, $\overline{G}(\m)$ is \emph{not\/} Gorenstein. 
\end{prop}

\begin{proof}
We first prove the following claim. 
\par \vspace{2mm} \par \noindent 
{\bf Claim 1.} $\lfloor \frac{k\delta}{a'}\rfloor=0$, 
$\lceil \frac{k\delta}{a'}\rceil=1$  for each $k$ $(1 \le k \le k_1-1)$. 
\par \vspace{2mm}
By definition, $k \le k_1-1< \frac{a'}{\delta}$. 
Thus $0 < \frac{k\delta}{a'} < 1$. 
Hence $\lfloor \frac{k\delta}{a'}\rfloor=0$ and 
$\lceil \frac{k\delta}{a'}\rceil=1$. 
\par \vspace{2mm} \par \noindent 
{\bf Claim 2.} $\lfloor \frac{k_1\delta}{a'}\rfloor=1$, 
$\lceil \frac{k_1\delta}{a'}\rceil=2$. 
\par \vspace{2mm}
By definition, $k_1-1 < \frac{a'}{\delta} \le k_1$. 
Since $(a',\delta)=1$, $\frac{a'}{\delta}$ is not an integer. 
Hence $\frac{a'}{\delta}< k_1$, that is, $\frac{k_1\delta}{a'} > 1$. 
If $\frac{k_1\delta}{a'} \ge 2$, then one has 
\[
\frac{2a'}{\delta}-1 \le k_1-1 < \frac{a'}{\delta}.  
\]
This yields $a' < \delta$, which is a contradiction. 
Hence $1 < \frac{k_1\delta}{a'}< 2$. 
The claim follows from here. 
\par \vspace{2mm} \par \noindent 
{\bf Claim 3.} $n_k=kn_1$ for every $k$ $(1 \le k \le k_1-1)$ and 
$n_{k_1}=k_1n_1+1$. 
\par \vspace{2mm} 
Indeed, for every $k$ $(1 \le k \le k_1-1)$, Claim 1 implies 
\[
n_k= \lfloor \frac{kb'}{a'} \rfloor = \lfloor \frac{k(n_1a'+\delta)}{a'}\rfloor
=kn_1+\lfloor \frac{k\delta}{a'}\rfloor=kn_1.
\]
On the other hand, Claim 2 implies 
\[
n_{k_1}=\lfloor \frac{k_1b'}{a'}\rfloor
=\lfloor \frac{k_1(n_1a'+\delta)}{a'}\rfloor
=k_1n_1+ \lfloor \frac{k_1\delta}{a'}\rfloor
=k_1n_1+1. 
\]
\par \vspace{2mm} \par \noindent 
{\bf Claim 4.} $n_{a-k}=(a-k)n_1+d\delta-1$ 
for each $k$ $(1 \le k \le k_1-1)$ and 
$n_{a-k_1}=(a-k_1)n_1+d\delta-2$. 
\par \vspace{2mm}
For every $k$ $(1 \le k \le k_1-1)$, Claim 1 implies 
\[
n_{a-k} 
= \text{\small 
$= \lfloor \frac{(a-k)(n_1a'+\delta)}{a'} \rfloor$ } 
= (a-k)n_1 + d \delta - \lceil \frac{k\delta}{a'} \rceil 
= (a-k)n_1 + d \delta-1. 
\]
On the other hand, Claim 2 implies 
\[
n_{a-k_1} 
=  \text{\small 
$= \lfloor \frac{(a-k_1)(n_1a'+\delta)}{a'} \rfloor $ }
= (a-k_1)n_1 + d \delta - \lceil \frac{k_1\delta}{a'} \rceil 
= (a-k_1)n_1 + d \delta-2. 
\]
\par \vspace{2mm} \par \noindent 
{\bf Claim 5.} $\ell_{t_1}+\ell_{r+1-t_1}=a+1$. 
 \par \vspace{2mm}
By Claim 3, $n_{k_1-1}+1=t_1 \le n_{k_1}$. 
Hence $\ell_{t_1}=k_1$ by Proposition \ref{LNlength}. 
By Claim 4, we have 
\begin{eqnarray*}
n_{a-k_1}+1 &=& (a-k_1)n_1+d\delta-1 = r+1-t_1 \\[2mm]
n_{a-k_1+1} &=& (a-k_1+1)n_1+d\delta-1 > r+1-t_1. 
\end{eqnarray*}
Hence $\ell_{r+1-t_1}=a-k_1+1$ and thus 
$\ell_{t_1}+\ell_{r+1-t_1}=(a-k_1+1)+k_1=a+1$.  
\end{proof}

\par 
Using Theorem \ref{Gor-Main}, we can determine all Zariski-type hypersurfaces having maximal relative normal reduction number.  

\begin{prop} \label{Hyp_Maximal}
Let $A=S/(x^a-g(\underline{y}))$ be a Zariski-type hypersurface with 
$b=\ord_{(\underline{y})}(g)$. 
Then 
\begin{enumerate}
\item[(a)] $\ol{r}(\m)=1$ if and only if $(a,b)=(2,2)$, $(2,3)$. 
\item[(b)] $\ol{r}(\m)=2$ if and only if $(a,b)=(2,4),(2,5),(3,3),(3,4)$. 
\item[(c)] $\nr(\m)$ is maximal if and only if $(a,b)=(2,n),(3,3),(3,4),(3,5)$, where $n \ge 2$ is any integer. 
\end{enumerate}
\end{prop}

\begin{proof}
Note that $\ol{r}(\m)=\lfloor \frac{(a-1)b}{a} \rfloor \ge a-1$ by Theorem \ref{Gor-Main}. 
\par \vspace{1mm} \par \noindent 
(a) $\ol{r}(\m) =1$ implies $a=2$ by the above remark. 
Thus one can easily see that  $\ol{r}(\m)=1$  if $(a,b)=(2,2)$, $(2,3)$. 
\par \vspace{1mm} \par \noindent 
(b) $\ol{r}(\m) =2$ implies $a=2,3$ by the above remark. 
When $a=2$, $\ol{r}(\m) =\lfloor \frac{b}{2} \rfloor =2$ if and only if $b=4,5$. 
When $a=3$, $\ol{r}(\m) =\lfloor \frac{2b}{3} \rfloor =2$ if and only if $b=3,4$. 
\par \vspace{1mm} \par \noindent 
(c) Suppose $r=\nr(m)$ is maximal. We may assume $r \ge 2$.  
By Theorem \ref{thm:nr(F)=r(F)}, 
we have $\ell(\overline{\m^n}/Q \overline{\m^{n-1}})=1$ 
for all $n=2,\ldots,r$.  
On the other hand, using Theorem \ref{Gor-Main},  we can 
show that  $\ell(\overline{\m^n}/Q \overline{\m^{n-1}}))=
\max\{a-\lceil \frac{an}{b}\rceil, 0\}$ 
for all $n \ge 2$. 
Thus 
\[
a-1=\lceil \frac{2a}{b}\rceil \le \lceil \frac{3a}{b}\rceil \le \cdots \le 
\lceil \frac{ra}{b}\rceil \le  a-1. 
\]
This means $\lceil \frac{2a}{b}\rceil =a-1$. 
Then $a-1 =\lceil \frac{2a}{b}\rceil \le 2$ yields $a \le 3$. 
By a similar argument as in the proof of $(b)$, we obtain the required result. 
It is not so difficult to chech that the converse is also true. 
\end{proof}

\par 
The following example gives a Zariski-type hypersurface 
such that $\nr(\m)$ is maximal but $\ols{G}(\m)$ is {\it not} Gorenstein. 
\begin{exam} \label{Ex3-5nonGor}
Let $m \ge 1$ be an integer, and  
let $A=\mathbb{C}[x, y_1, \ldots,y_m]/(x^3+y_1^5+\cdots+y_m^5)$. 
Put $\m = (x, y_1, \ldots, y_m)$. Then 
\begin{enumerate}
\item[\rm{(a)}] $A_{\m}$ is a rational singularity if and only if $m \ge 4$.  
\item[\rm{(b)}] $\overline{\m^2}=\m^2$, $\overline{\m^3}=\m^3+(x^2)$, 
and $\overline{\m^4}= (y_1, y_2, \ldots, y_m) \overline{\m^3}$. 
\item[\rm{(c)}] $\overline{\mathcal{R}'}(\m)=K[xt,y_1t,y_2t,\ldots,y_mt, x^2t^3,t^{-1}]$. 
\item[\rm{(d)}] $\overline{G}(\m)$ is a Cohen-Macaulay ring of $e_0(\overline{G})=3$ and CM type $2$. 
\end{enumerate} 
If we put $X=xt$, $Y_i=y_it$ $(i=1,2,\ldots,m)$ and $Z=x^2t^3$, and $U=t^{-1}$, 
then 
\[
\overline{\mathcal{R}'}(\m) \cong 
K[X,\underline{Y},Z,U]/(X^2-ZU, XZ+\underline{Y^5}U, Z^2+X\underline{Y^5}). 
\]
In particular, 
\[
\overline{G}(\m) \cong 
\overline{\mathcal{R}'}(\m)/t^{-1}\overline{\mathcal{R}'}(\m)
\cong K[X,\underline{Y},Z]/(X^2, XZ, Z^2+X\underline{Y^5}),
\]
where $\underline{Y^5}=Y_1^5+\cdots+Y_m^5$. 
\end{exam}

\begin{exam}
Let $A = K \llbracket t^a, t^b \rrbracket$ be a semigroup ring generated by the positive integers $a \leq b$ such that $\gcd(a,b)=1$. We know that $A \cong k\llbracket x, y \rrbracket/(x^a- y^b)$ is a hypersurface. By Theorem \ref{S1-Main} $\ol{r}_{Q}(\m) = \lfloor \frac{(a-1)b}{a}\rfloor$. 
By Theorem \ref{Gor-Main}, $\ol{G}(\m)$ is Gorenstein if and only if $b \equiv  1 \pmod{a}$.
\end{exam}
\par 
The examples \ref{Exam:3and4} and \ref{Exam:2andm} 
can also be verified using Theorem \ref{Gor-Main}.

\bigskip
\section{Normal Tangent cones of maximal embedding dimension}\label{sec:embdimGm}

\par \vspace{2mm}
Let $\MM_{\overline{G}}$ denote the unique graded maximal ideal of 
$\overline{G}=\overline{G}(\m)$. 
Then $\emb(\overline{G})$, the \textit{embedding dimension} of $\overline{G}$,  
is the number of minimal system of generators of $\MM_{\overline{G}}$. 
Since $\overline{G}$ is Cohen-Macaulay, 
\begin{equation} \label{Sally}
\emb(\overline{G}) \le e_0(\overline{G})+\dim \overline{G}-1=a+m-1
\end{equation}
If equality holds in Eq.(\ref{Sally}), then $\overline{G}$ is said to 
have \textit{maximal embedding dimension} (in the sense of Sally). 
Note that $\MM_{\overline{G}}$ is generated by 
$\{x^kt^{n_k}\}_{k=1,2,\ldots,a-1}$, and $y_1t,y_2t,\ldots, y_m t$.

\par \vspace{2mm}
We can characterize those $\overline{G}$ in terms of $a$ and $b$. 

\begin{prop} \label{MaxEmb}
Let $A=S/(x^a-g(\underline{y}))$ be a Zariski hypersurface with 
$b=\ord_{(\underline{y})}(g) \ge a$.   
Then $\overline{G}(\m)$ has maximal embedding dimension 
if and only if 
one of the following cases occurs$:$
\begin{enumerate}
\item[\rm{(a)}] $a=2$. 
\item[\rm{(b)}] $a \ge 3$,  and $b\equiv a-1 \pmod{a}$.  
\end{enumerate}
When this is the case, 
$\overline{G}(\m)$ is Gorenstein if and only if $a=2$. 
\end{prop}

\begin{proof}
In the case of $a=2$, since $\overline{G}$ is a hypersurface of $e_0(\overline{G})=2$, it is a Gorenstein ring having maximal embedding dimension. So we may assume that $a \ge 3$. 
\par \vspace{2mm} \par \noindent 
$\rm{(b)} \Longrightarrow \rm{(a)}:$ 
We may compute $n_k$ as follows:
\[
n_k
= \lfloor \frac{k(n_1a+a-1)}{a} \rfloor
=kn_1+k+\lfloor \frac{-k}{a} \rfloor =kn_1+k-\lceil \frac{k}{a} \rceil
= kn_1+k-1
\]
for every $k=1,2,\ldots,a-1$. 
Then a sequence 
\[
xt^{n_1},\,x^2t^{2n_1+1},\, x^{3}t^{3n_1+2}, \, \ldots, \, 
x^{a-1}t^{(a-1)n_1+a-2},\; y_1t,\ldots,y_mt
\]
forms a minimal system of generators of $\MM_{\overline{G}}$, we have 
\[
\emb(\overline{G})=a+m-1=e_0(\overline{G})+\dim \overline{G}-1. 
\]
$\rm{(a)} \Longrightarrow \rm{(b)}:$ 
In order to complete the proof, we prove the following two lemmata.

\begin{lem} \label{MaxEmb-1}
If $d>1$ and $a \ge 3$, then $\overline{G}$ does not have maximal embedding dimension. 
\end{lem}

\begin{proof}
We note that $\MM_{\overline{G}}=(\{x^kt^{n_k}\}_{k=1,2,\ldots,a-1},
y_1t,\ldots,y_mt)$. 
For every $s$ $(1 \le s \le d)$, $\delta$ $(1 \le \delta \le a'-1)$, 
we have 
\[
n_{sa'+\delta}=\lfloor \frac{(sa'+\delta)b}{a} \rfloor
=\lfloor \frac{(sa'+\delta)b'}{a'} \rfloor
=sb'+\lfloor \frac{\delta b'}{a'} \rfloor
=sn_{a'}+n_{\delta}, 
\]
and hence 
\[
x^{sa'+\delta}t^{n_{sa'+\delta}} = x^{sa'+\delta}t^{sn_{a'}+n_{\delta}}
=(x^{a'}t^{n_{a'}})^s \cdot (x^{\delta}t^{n_{\delta}}) \in \MM_{\overline{G}}^2. 
\]
Moreover, if $s \ge 2$, then 
\[
x^{sa'}t^{n_{sa'}} = x^{sa'}t^{sn_{a'}} = (x^{a'}t^{n_{a'}})^s \in   \MM_{\overline{G}}^2. 
\]
Hence $\MM_{\overline{G}}=(\{x^kt^{n_k}\}_{k=1,2,\ldots,a'},y_1t,\ldots,y_mt)$.
Since $a \ge 3$ and $d \ge 2$, we obtain that 
\[
\emb(\overline{G})\le a'+m \le \frac{a}{2}+m < (a-1)+m=a+m-1, 
\]
as required. 
\end{proof}

\begin{lem} \label{MaxEmb-2}
If $d=1$ and $1 \le \delta \le a-2$, then $\overline{G}$ does not have maximal embedding dimension. 
\end{lem}

\begin{proof}
For each $k$ $(1 \le k \le a-1)$, if we put $m_k=n_k-kn_1$, we have 
\[
m_k= n_{k}-kn_1 =\lfloor \frac{kb}{a} \rfloor -kn_1
=\lfloor \frac{k(n_1a+\delta)}{a} \rfloor -kn_1
=\lfloor \frac{k \delta}{a} \rfloor. 
\]
Since $1 \le \delta \le a-2$ and $a \ge 3$, we get 
\[
m_{a-1} = \lfloor \frac{(a-1)\delta}{a} \rfloor
\le \lfloor \frac{(a-1)(a-2)}{a} \rfloor
=a-3+ \lfloor \frac{2}{a} \rfloor=a-3. 
\]
Then since $0 \le m_1 \le m_2 \le \cdots \le m_{a-1} \le a-3$, 
there exists an integer $i$ with $2 \le i \le a-1$ such that 
$m_{i-1}=m_i$ by Pigeonhole principle. 
Since $n_1+n_{i-1}=n_1+\left\{(i-1)n_1+m_{i-1} \right\} = in_1+m_i=n_i$, 
we get
\[
x^it^{n_i} = x^it^{n_1+n_{i-1}}=(xt^{n_1})(x^{i-1}t^{n_{i-1}}) \in \MM_{\overline{G}}^2. 
\]
This implies $\emb(\overline{G}) \le a+m-2 < e_0(\overline{G})+\dim \overline{G}-1$, 
that is, $\overline{G}$ does not have maximal embedding dimension. 
\end{proof}

\par \vspace{3mm}
By the lemmata as above, we complete the proof of Proposition \ref{MaxEmb}. 
\end{proof}

\section*{Acknowledgements}
The authors thank M. E. Rossi for helpful discussions.
Moreover, the authors are grateful to the referee for the careful reading and valuable coments.

\bibliography{v3}
\bibliographystyle{abbrv}

\end{document}